\documentclass{amsart}
\usepackage{amsfonts,amsmath,amsthm,amssymb,tikz,xcolor,mathtools,extarrows,graphicx,stackrel}
\usepackage[normalem]{ulem}
\useunder{\uline}{\ul}{}
\usepackage[margin=1in]{geometry}
\usepackage[colorlinks=true,linkcolor=red!70!black,citecolor=green!70!black,urlcolor=magenta!70!black,backref]{hyperref}
\usepackage{colortbl}
\usepackage{hhline}
\usepackage{xcolor}
\usepackage{dsfont}
\usepackage{standalone}
\usepackage{comment}
\usepackage{wasysym} % for full and new moons
\usepackage{bbm} % for blackboard bold 1, k...
\usepackage{stmaryrd} % for Kh brackets
\usepackage{enumitem}
\usepackage{xargs}
\usepackage{color}
\usepackage[all, color, cmtip]{xy}
\usepackage{mathtools}
\usepackage{braket}
\usepackage[capitalize,nameinlink]{cleveref}
\usepackage{float} 
\newcommand{\hackcenter}[1]{\xy (0,0)*{#1}; \endxy}
\newcommand{\und}{\underline}

\colorlet{darkblue}{blue!70!black}
\colorlet{darkred}{red!70!black}
\colorlet{darkgreen}{green!70!black}
\colorlet{darkwhite}{white!65!black}
\colorlet{darkorange}{orange!70!black}

\colorlet{darkmagenta}{magenta!70!black}
\colorlet{pink}{green!20!magenta}

\colorlet{lightcyan}{cyan!15!white}
\colorlet{lightyellow}{yellow!30!white}
\colorlet{lightcyanb}{lightcyan!50!black}

\definecolor{purple}{rgb}{0.63, 0.36, 0.94} % official vertex color

\usetikzlibrary{decorations.markings,decorations.pathreplacing, decorations.pathmorphing, calligraphy,positioning,calc,arrows.meta,cd}

\numberwithin{equation}{section}

\newtheorem{thm}{Theorem}[section]
\newtheorem{lemma}[thm]{Lemma}
\newtheorem{proposition}[thm]{Proposition}
\newtheorem{corollary}[thm]{Corollary}
\newtheorem{cor}[thm]{Corollary}

\newtheorem{example}[thm]{Example}

\theoremstyle{definition}
\newtheorem{definition}[thm]{Definition}

\newtheorem{remark}[thm]{Remark}

\usepackage{bbm}
\def\C{{\mathbb{C}}}

\def\Z{{\mathbbm Z}}

 \def\1{\mathbbm{1}}%

\def\Q{{\mathbbm Q}}
\renewcommand{\to}{\rightarrow}
\newcommand{\maps}{\colon}

\renewcommand{\Im}{\mathrm{im}}
\renewcommand{\S}{Section }
\begin{document}

\title{Spectral geometry of Khovanov Laplacians}
 
\author[J. Grlj]{Jernej Grlj}
\address{Department of Mathematics\\
 University of Southern California \\
  Los Angeles, California 90089, USA}
  \email{grlj@usc.edu}
\author[A.D. Lauda]{Aaron D. Lauda}
\address{Department of Mathematics\\
 University of Southern California \\
  Los Angeles, California 90089, USA}
  \email{lauda@usc.edu}
\date{\today}

\begin{abstract}
For an oriented link diagram $D$, the Khovanov cochain complex carries a canonical Hermitian inner product that defines a combinatorial Hodge Laplacian. Its kernel is naturally isomorphic to Khovanov cohomology over $\C$, while its positive spectrum depends on the chosen diagram. Building on the numerical work of Jones and Wei, we develop a structural study of this diagram-dependent higher spectrum. In the minimal and maximal cube degrees, we identify the Khovanov Laplacian with signless Laplacians of explicit weighted graphs, up to diagonal sign conjugation in maximal degree. The graph model describes rational Khovanov classes and harmonic representatives through its bipartite components and gives inverse-polynomial gap bounds in fixed-width bands near the extremal $q$-degrees when the corresponding rational cohomology vanishes.  It also yields exact $\Theta(N^{-2})$ bidegree gaps for twisted unknots, odd $(2,N)$-torus knots, and even twist knots. The exact gap calculations are motivated in part by the spectral-resolution requirement in a recent quantum algorithm for Khovanov homology. We also prove a finite-dimensional analytic--combinatorial torsion correspondence for the Khovanov complex showing that an alternating product of nonzero Laplacian pseudodeterminants recovers integral Khovanov torsion in rationally acyclic $q$-degrees and differs from it by an explicit regulator in general. Finally, we transfer Lee's deformation to a filtered differential on the harmonic Khovanov subspace, giving a canonical harmonic model for the Lee spectral sequence.
\end{abstract}
\maketitle

\section{Introduction}

A standard way to extract structure from a space is to study the spectrum of an associated Laplacian. On a closed Riemannian manifold, the kernel of the Hodge--de Rham Laplacian consists of harmonic forms and computes de Rham cohomology, while the positive spectrum governs diffusion and reflects geometric features that cohomology alone cannot see. The same division appears in the discrete setting. Eckmann's combinatorial Laplacian has kernel isomorphic to rational cohomology~\cite{Eckmann1945}, while its positive eigenvalues record properties of the chosen combinatorial model, with applications ranging from spectral graph theory~\cite{Fiedler1973,chung1997spectral} to persistent Laplacians in topological data analysis~\cite{WangNguyenWei2020,Mmoli2022,WeiWei2025}. A round sphere and a sufficiently elongated ellipsoid, for example, have the same de Rham cohomology but different low-frequency spectra. In this familiar setting, the kernel sees topology, while the positive spectrum sees geometry.

This paper develops an analogous spectral viewpoint for Khovanov homology. Khovanov's categorification of the Jones polynomial~\cite{Khovanov2000} assigns to an oriented link diagram $D$ a finite-dimensional bigraded cochain complex $C_{\mathrm{KH}}(D)$ whose cohomology is an invariant of the underlying link $K$. We equip the enhanced-state basis with the Hermitian inner product for which the tensor-product basis in $1$ and $x$ is orthonormal. The resulting Khovanov Laplacian is
\begin{equation}\label{eq:intro-laplacian}
\Delta_i^{\mathrm{KH}}(D)
:=
\partial_i^{\dagger}\partial_i
+
\partial_{i-1}\partial_{i-1}^{\dagger},
\end{equation}
and we write $\Delta_{i,j}^{\mathrm{KH}}(D)$ for its restriction to the $q$-degree-$j$ summand. Combinatorial Hodge theory gives canonical identifications
\begin{equation}\label{eq:intro-hodge}
\ker\!\left(\Delta_{i,j}^{\mathrm{KH}}(D)\right)
\cong
\mathrm{KH}^{i,j}(K;\C).
\end{equation}
Thus each complex Khovanov class has a distinguished harmonic representative, and the multiplicity of the zero eigenvalue depends only on the link.

The operator~\eqref{eq:intro-laplacian} was introduced and studied numerically by Jones and Wei~\cite{Jones2025}. Their computations made clear that the positive spectrum depends on the chosen diagram. The crossingless unknot has zero differential and no positive Laplacian spectrum, while a one-crossing Reidemeister~I diagram of the same unknot has nonzero eigenvalues, see Example \ref{exam:twisted}. More generally, Reidemeister chain homotopy equivalences preserve cohomology but need not preserve adjoints, orthogonal complements, or Laplacian eigenvectors.  The Khovanov spectrum should therefore be viewed as spectral data attached to the cochain complex of a knot diagram, rather than the knot itself.

This diagram dependence begs the question of what information is encoded in the positive spectrum. Building on the numerical work of Jones and Wei, we develop a systematic structural study of the higher spectrum of the Khovanov Laplacian. At the bottom of the spectrum, we identify graph-theoretic mechanisms that govern harmonic representatives and small eigenvalues in the extreme cube degrees. At the level of the full positive spectrum, an alternating determinant recovers integral information that is absent from the complex Hodge kernel. The Hodge contraction also transfers Lee's deformation to the harmonic subspace, placing the Lee spectral sequence on canonical Khovanov representatives.

For a positive semidefinite operator $L$, write $\operatorname{gap}(L)$ for its smallest positive eigenvalue. In a Khovanov bidegree this has the Rayleigh-quotient description
\begin{equation}\label{eq:intro-gap}
\operatorname{gap}\!\left(\Delta_{i,j}^{\mathrm{KH}}(D)\right)
=
\min_{\substack{0\neq v\perp\ker\Delta_{i,j}^{\mathrm{KH}}(D)}}
\frac{\langle v,\Delta_{i,j}^{\mathrm{KH}}(D)v\rangle}{\langle v,v\rangle}
=
\min_{\substack{0\neq v\perp\ker\Delta_{i,j}^{\mathrm{KH}}(D)}}
\frac{\|\partial v\|^2+\|\partial^{\dagger}v\|^2}{\|v\|^2}.
\end{equation}
It measures how close a non-harmonic cochain can come to being both closed and coclosed. Our first results give a concrete description of this low-energy problem in the minimal and maximal cube degrees.

\subsection*{The low spectrum in the extreme cube degrees}

Graph Laplacians have appeared previously in knot theory through the coloring Laplacians of Silver and Williams~\cite{Silver2019}, and graph-theoretic properties of link diagrams also control parts of extremal Khovanov homology, as in the work of Sazdanovi\'c and Scofield on graph girth~\cite{Sazdanovi2022}. The graphs used here serve a different purpose: they give an exact matrix model for the Khovanov Laplacian itself in the extreme cube degrees.

In cube degree zero, we construct a finite weighted graph $G_{0,q}(D)$ such that
\begin{equation}\label{eq:intro-signless}
\widetilde\Delta_{0,q}(D)
=
Q\bigl(G_{0,q}(D)\bigr),
\qquad
Q(G)=D_G+A_G,
\end{equation}
where $Q(G)$ is the signless graph Laplacian, $A_G$ is the weighted adjacency matrix, and $D_G$ is the weighted degree matrix. In maximal cube degree, the corresponding Khovanov Laplacian is diagonally conjugate to a signless graph Laplacian (the conjugating signs account for the cube signs in the Khovanov differential). This identification converts the extreme-degree Hodge problem into spectral graph theory. Standard results for signless Laplacians~\cite{HAEMERS2004199,Desai} imply that loop-free bipartite components of $G_{0,q}(D)$ correspond to Khovanov classes over $\Q$ in the shifted bidegree, and a two-coloring of such a component produces the associated harmonic representative as a signed sum of vertices; see Theorem~\ref{thm:khovanov-bipartite} and Corollary~\ref{cor:harmonic-2-coloring}.

The same graph model gives quantitative control of the bottom eigenvalue. Suppose the rational Khovanov group in the corresponding shifted bidegree vanishes. Let $\ell_0$ be the number of circles in the all-zero resolution, and let $r$ be the number of labels that differ from a chosen extremal labeling. The graph in that $q$-degree has
\[
\lvert V(G_{0,q}(D))\rvert=\binom{\ell_0}{r}.
\]
Because the smallest positive edge or loop weight in these graphs is $1/2$, the Desai--Rao estimate gives
\begin{equation}\label{eq:intro-general-gap-bound}
\operatorname{gap}\!\left(\widetilde\Delta_{0,q}(D)\right)
\geq
\frac{1}{16d^*\lvert V(G_{0,q}(D))\rvert^2},
\end{equation}
where $d^*$ is the maximal weighted vertex degree. For a family of $N$-crossing diagrams with $\ell_0=O(N)$, this is an inverse-polynomial lower bound in every fixed-width band near an extremal $q$-degree, since $r$ is then bounded independently of $N$. In the middle $q$-degrees, however, $r$ can be of order $\ell_0/2$ and the binomial coefficient can be exponential in $N$. The Desai--Rao estimate alone therefore gives no inverse-polynomial lower bound in the bulk of the grading range.

There is also a concrete reason to seek exact asymptotics for these small eigenvalues. The quantum algorithm for Khovanov homology developed by Schmidhuber, Reilly, Zanardi, Lloyd, and the second author uses phase estimation to separate the harmonic subspace from the positive spectrum~\cite{schmidhuber2025quantumalgorithmkhovanovhomology}. Its spectral resolution therefore depends on the gap in the target bidegree. The numerical computations in that work single out natural diagram families and bidegrees in which the minimum gap appears to occur. This motivates an exact analysis of the corresponding low-energy matrices.

For the standard $N$-crossing diagrams considered in Sections~\ref{ssec:twisted-unknot-graphs}--\ref{ssec:even-twist-knots}, we obtain
\begin{align*}
\operatorname{gap}\!\left(\widetilde\Delta_{0,3-N}(TU_N)\right)
&=4\sin^2\!\left(\frac{\pi}{2(N+2)}\right),\\
\operatorname{gap}\!\left(\Delta_{N-1,3N-2}^{\mathrm{KH}}(T_{(2,N)})\right)
&=4\sin^2\!\left(\frac{\pi}{2N}\right)
\qquad (N\ \text{odd}),\\
\operatorname{gap}\!\left(\Delta_{N-2,3N-9}^{\mathrm{KH}}(\mathrm{Tw}_N)\right)
&=4\sin^2\!\left(\frac{\pi}{2(N-1)}\right)
\qquad (N\ \text{even}),
\end{align*}
for twisted unknots $TU_N$, $(2,N)$-torus knots $T_{(2,n)}$, and even-twist knots $Tw_N$. Each gap is $\Theta(N^{-2})$; see Theorems~\ref{thm:tu-gap}, \ref{thm:torus-spectral-gap}, and~\ref{thm:even-twist-knots}. These formulas establish an inverse-polynomial spectral resolution scale in the bidegrees selected by the numerical data.  But we do not prove that these bidegrees contain the global minimum for the corresponding diagram or obtain a uniform polynomial lower bound for all knots or all diagrams needed for the quantum algorithm.  We also do not address the state preparation and readout requirements of that algorithm.   

\subsection*{Integral information in the full nonzero spectrum}

The spectral gap records where the positive spectrum begins.   Classical torsion theory provides a model for how the nonzero spectrum can be packaged into a determinant like quantity.  
Reidemeister and Franz used combinatorial torsion to distinguish spaces with the same homology, most famously lens spaces~\cite{Reidemeister1935,Franz1935}. Ray and Singer later introduced analytic torsion as a spectral counterpart, defined through zeta-regularized determinants of Hodge Laplacians~\cite{RaySinger1971}. Cheeger and M\"uller proved that analytic and Reidemeister torsion agree for closed manifolds in the classical acyclic setting~\cite{Cheeger1979,Muller1978}. When cohomology is present, the comparison includes a regulator measuring the covolume of the integral cohomology lattice; this is part of the refined picture developed by Bismut and Zhang~\cite{BismutZhang1992}.

For a finite-dimensional cochain complex, no zeta regularization is required. The analytic expression used to define torsion is an ordinary alternating product of Laplacian pseudo-determinants (products of nonzero eigenvalues). In a fixed $q$-degree we set
\begin{equation}\label{eq:intro-tau}
\tau_j^{\C}(D)
=
\prod_i
\det\!{}'\!\left(\Delta_{i,j}^{\mathrm{KH}}(D)\right)^{(-1)^{i+1}i/2}.
\end{equation}
A finite-dimensional form of the analytic--combinatorial torsion correspondence, as developed for based complexes by Mnev~\cite{mnev2014lecturenotestorsions}, applies directly to the integral Khovanov complex. When computed in orthonormal bases Ortiz-Navarro's Khovanov volume form~\cite{ORTIZNAVARRO2012} agrees with $\tau_j^\C$. Our result identifies its integral meaning and makes the cohomological regulator explicit, justifying the computational observations from that work; see Remark \ref{rmk:ortiz_navarro_conj}.
 
More precisely, let $\tau_j^{\Z}(K)$ denote the torsion determined by the integral Khovanov cohomology, and let $V_j(g,f)$ be the covolume of a basis $g$ for the free part of the integral cohomology lattice measured in an orthonormal harmonic basis $f$. In the conventions of Section~\ref{sec:torsion-knots},
\begin{equation}\label{eq:intro-torsion-regulator}
\tau_j^{\C}(D)
=
\tau_j^{\Z}(K)\,V_j(g,f).
\end{equation}
When $\mathrm{KH}^{*,j}(K;\C)=0$, there is no regulator and the nonzero spectrum directly recovers the alternating product of the orders of the integral torsion subgroups. When rational cohomology is present, the regulator records the difference between the integral lattice and the orthonormal harmonic basis. Thus the individual positive eigenvalues remain diagram-dependent, while their alternating determinant recovers precise integral information after the lattice correction.

This shows that thought we work with rational (or complex) Khovanov homology, the higher spectrum packaged into combinatorial torsion still detects information coming from integral homology, even though the higher spectrum is diagram dependent.   

\subsection*{Filtered structure on harmonic Khovanov representatives}

Hodge theory provides more than the dimension of Khovanov cohomology: it selects a canonical representative of every complex class and supplies a contraction of the full cochain complex onto the harmonic subspace. Lee's deformation gives a natural setting in which to use this additional structure. Lee replaces the Khovanov differential by a filtered perturbation whose associated spectral sequence begins with rational Khovanov homology and converges to Lee homology~\cite{Lee2005}. The resulting filtration underlies Rasmussen's $s$-invariant and its bound on the smooth slice genus~\cite{Rasmussen2010}.

Let $\mathcal H=\ker\Delta^{\mathrm{KH}}$, let $P$ be orthogonal projection onto $\mathcal H$, let $\iota\colon\mathcal H\hookrightarrow C_{\mathrm{KH}}(D)$ be the inclusion, and let $h=\partial^{\dagger}G$ be the Hodge contracting homotopy defined using the Green operator $G$. For the Lee perturbation
\[
\Phi=\partial^{\mathrm{Lee}}-\partial^{\mathrm{KH}},
\]
the homological perturbation lemma transfers the Lee differential to
\begin{equation}\label{eq:intro-lee}
D_{\mathcal H}
=
P\Phi\sum_{r\geq 0}(-1)^r(h\Phi)^r\iota.
\end{equation}
The complex $(\mathcal H,D_{\mathcal H})$ is chain homotopy equivalent to the Lee complex, and the homogeneous terms of~\eqref{eq:intro-lee} compute the page differentials of the Lee spectral sequence; see Theorem~\ref{thm:lee-hodge-transfer} and the explicit perturbation framework of~\cite{Romero2026}. Proposition \ref{prop:lee-generators-not-kh-harmonic} shows that Lee's standard chain representatives are not generally harmonic for the Khovanov Laplacian, so this construction cannot be replaced by restricting the Lee differential to $\mathcal H$.

The transfer uses the Green operator, which inverts the Khovanov Laplacian on the positive spectral subspace. In this sense, the Lee model depends on the full Hodge contraction rather than only on the kernel. We leave for further work to extract an explicit efficient algorithm for computing $s$ in the harmonic realization.

\subsection*{Organization}

Section~\ref{sec:hodge} reviews finite-dimensional Hodge theory, Green operators, and homological perturbation. Section~\ref{sec:khovanov-background} fixes the Khovanov conventions and the Hermitian inner product used to define the Laplacian, including the diagram dependence of its positive spectrum. Section~\ref{sec:graph} develops the signless graph model, its structural consequences, the fixed-width extreme-degree gap bound, and the exact calculations for twisted unknots, odd $(2,N)$-torus knots, and even twist knots. Sections~\ref{sec:torsion-general} and~\ref{sec:torsion-knots} establish the finite-dimensional torsion correspondence and apply it to integral Khovanov homology. Section~\ref{sec:lee} transfers Lee's deformation to the harmonic Khovanov subspace. Section~\ref{sec:future} collects open problems.

\subsection*{Acknowledgments}

The authors are grateful to Anthony Licata for helpful conversations and for pointing out the connection between combinatorial Laplacians and Reidemeister torsion. We thank ICERM for its hospitality, where these conversations first took place. The authors thank Melissa Zhang for discussions related to spectral sequences of knot homologies. The authors thank Alexander Schmidhuber, Michele Reilly, Paolo Zanardi, and Seth Lloyd for helpful discussions where the extreme homological graph bounds were first developed. A.D.L.\ and J.G.\ were partially supported by NSF grant DMS-2200419 and DMS-2601214 and the Simons Foundation Collaboration Grant on New Structures in Low-Dimensional Topology MPS-LDT-00920756 and SFI-MPS-LDT-00014760-01. Computations associated with this project were conducted using the Center for Advanced Research Computing (CARC) at the University of Southern California and made use of Bar-Natan's KnotTheory package together with the knot encodings from KnotInfo.

\subsection*{On the use of AI models}
Large Language Models have been used to assist in the proof of Theorem \ref{thm:tu-gap}, help with literature search for Section \ref{sec:lee}, and to assist with writing and proofreading. All errors are ours.

%-------------------------------------------
\section{Combinatorial Hodge theory}\label{sec:hodge}
%-------------------------------------------
In this section we collect the combinatorial Hodge theory used in the rest of the paper, in the form developed by Eckmann~\cite{Eckmann1945} and Friedman~\cite{MR1622290}. The results hold over any field of characteristic zero. We work over $\C$.

Let $ C =(C^i,d_i)_{i\in\Z}$ be a cochain complex of finite-dimensional complex vector spaces in which each cochain space $C^i$ is equipped with an inner product $\langle\cdot,\cdot\rangle$. The adjoint $d_i^{\dagger}\maps C^{i+1}\to C^i$ is defined by
\[
\langle d_i(c),c'\rangle=\langle c,d_i^{\dagger}(c')\rangle
\]
for $c\in C^i$ and $c'\in C^{i+1}$. The differentials and their adjoints give rise to the combinatorial Laplacian
\begin{equation}
\Delta_i:=d_i^{\dagger}d_i+d_{i-1}d_{i-1}^{\dagger}.
\end{equation}
We call an $i$-cochain \emph{harmonic} if $\Delta_i c=0$. The subspace of harmonic $i$-cochains is denoted
\begin{equation}
\mathcal H^i:=\{c\in C^i\mid \Delta_i c=0\}.
\end{equation}

The summands of $\Delta_i$ have complementary support: $d_i^{\dagger}d_i$ annihilates $\mathcal H^i$ and $\Im(d_{i-1})$ while preserving $\Im(d_i^{\dagger})$, and $d_{i-1}d_{i-1}^{\dagger}$ annihilates $\mathcal H^i$ and $\Im(d_i^{\dagger})$ while preserving $\Im(d_{i-1})$.

The basic Hodge theorem in this setting~\cite[Propositions 2.1--2.2]{MR1622290} reads as follows.
\begin{thm}[Hodge decomposition]\label{thm_Hodge}
Given a complex $ C $ equipped with an inner product, each cochain space $C^i$ decomposes as
\begin{equation}\label{eq:Cidecomp}
C^i\cong \mathcal H^i\oplus\Im(d_{i-1})\oplus\Im(d_i^{\dagger}).
\end{equation}
Furthermore, there is an isomorphism $\mathcal H^i\cong \mathrm H^i( C )$, so each cohomology class is represented by a unique harmonic cochain in $\mathcal H^i$.
\end{thm}

In particular, Hodge theory implies that $\Delta_i$ is positive definite on the orthogonal complement
\[
(\mathcal H^i)^{\perp}=\Im(d_{i-1})\oplus\Im(d_i^{\dagger}).
\]
We write $\operatorname{gap}(\Delta_i)$ for the smallest positive eigenvalue of $\Delta_i$, whenever the positive spectrum is nonempty.

\begin{remark}\label{rmk:eigenvalue-inheritance}
Every nonzero eigenvalue of $\Delta_i$ on $(\mathcal H^i)^{\perp}$ is also an eigenvalue of either $\Delta_{i-1}$ on $(\mathcal H^{i-1})^{\perp}$ or $\Delta_{i+1}$ on $(\mathcal H^{i+1})^{\perp}$~\cite[Proposition 2.2]{MR1622290}, and occurs as an eigenvalue of either $d_i^{\dagger}d_i$ or $d_{i-1}d_{i-1}^{\dagger}$ on its respective image.
\end{remark}

In general, chain maps do not preserve the Hodge decomposition; for an explicit example see~\cite[Remark 3.1]{Liu}. Section~\ref{subsec:Hom-pert} below describes a modification that recovers a well-defined action on harmonic cochains.

%  - - - - -- - --  - -- -  --  - - -
\subsection{Strong deformation retracts and harmonic representatives}\label{subsec:Hom-pert}
%  - - - - -- - --  - -- -  --  - - -
Chain maps between inner-product complexes do not in general preserve the Hodge decomposition. We now describe a standard modification, following Liu--Li--Wu~\cite{Liu}, that yields a well-defined functor on harmonic cochains and supplies the contracting homotopy used in Section~\ref{sec:lee} to transfer Lee's perturbation to the harmonic Khovanov subspace.

\begin{definition}\label{def:SDR}
A chain map $\pi\maps C \to\mathcal D$ is a \emph{strong deformation retract} if there exist a chain map $\iota\maps\mathcal D\to C $ and a degree $-1$ map $h\maps C \to C $ satisfying
\[
\pi\iota=\mathrm{Id}_{\mathcal D},\qquad
\mathrm{Id}_{ C }-\iota\pi=dh+hd,\qquad
h\iota=0.
\]
The data $( C ,\mathcal D,\pi,\iota,h)$ is then a \emph{strong deformation retraction} of $ C $ onto $\mathcal D$, and $\iota$ is the \emph{inclusion in a strong deformation retract}. Without loss of generality we may further assume $\pi h=0$ and $h^2=0$.
\end{definition}

Consider the chain complex $\mathcal H=(\mathcal H^i,0)$ consisting of the harmonic $i$-cochains with trivial differential. The Hodge decomposition
\[
C^i=\Im(d_{i-1})\oplus\Im(d_i^{\dagger})\oplus\mathcal H^i
\]
applied to the cochain spaces of $ C =(C^i,d_i)$ gives an inclusion chain map $\iota\maps\mathcal H\to C $. By Theorem~\ref{thm_Hodge}, $\iota$ is a quasi-isomorphism: it induces an isomorphism
\[
\iota^*\maps \mathrm H^i(\mathcal H)=\mathcal H^i\to \mathrm H^i( C ).
\]

Over a field, any quasi-isomorphism extends to a strong deformation retraction. Define a projection chain map $\pi\maps C \to\mathcal H$ by
\[
\pi(\Im(d_{i-1}))=\pi(\Im(d_i^{\dagger}))=0,
\qquad
\pi(x)=x\quad\text{for }x\in\mathcal H^i,
\]
and linear maps $h\maps C^i\to C^{i-1}$ assembling into a chain homotopy $h\maps C \to C $ by
\[
h(\mathcal H^i)=h(\Im(d_i^{\dagger}))=0,
\qquad
h(d_{i-1}d_{i-1}^{\dagger}x)=d_{i-1}^{\dagger}x.
\]
The latter formula completely determines $h$, since every nonzero element of $\Im(d_{i-1})$ has the form $d_{i-1}d_{i-1}^{\dagger}x$. Set $P=\iota\pi$. Then $P^2=P$, and $P$ acts as zero on non-harmonic cochains and as the identity on $\mathcal H$.

\begin{proposition}[{\cite[Section 3]{Liu}}]\label{prop:defret}
The data $( C ,\mathcal H,\pi,\iota,h)$ defines a strong deformation retraction of $ C $ onto $\mathcal H$. In particular, $\pi\iota=\mathrm{Id}_{\mathcal H}$, $h^2=h\iota=\pi h=0$, and $\mathrm{Id}_{ C }-\iota\pi=dh+hd$.
\end{proposition}

\begin{proof}
The argument of~\cite[Section 3]{Liu}, written there for differential graded inner-product spaces, applies verbatim.
\end{proof}

The chain homotopy $h$ refines the Hodge decomposition itself, in a way we record for later use.

\begin{proposition}[{\cite[Proposition 3.3]{Liu}}]\label{prop:Kdecomp}
The identity $\mathrm{Id}_{ C }=hd+dh+P$ refines the Hodge decomposition into
\[
 C =dh( C )\oplus hd( C )\oplus P( C ),
\]
with $P( C )=\mathcal H=\ker d\cap\ker d^{\dagger}$, $dh( C )=d( C )$, and $hd( C )=d^{\dagger}( C )$.
\end{proposition}

The decomposition in Proposition~\ref{prop:Kdecomp} is what makes the next result possible: although chain maps need not preserve the Hodge decomposition, orthogonal projection recovers a well-defined action on harmonic cochains.

\begin{proposition}\label{prop:h-functor}
For an inner-product complex $ C $, write $\mathsf H( C )=\ker d\cap\ker d^{\dagger}$. Given a chain map $f\maps C \to\mathcal D$, define
\begin{align}
\mathsf H(f)\maps\mathsf H( C )&\to\mathsf H(\mathcal D),\\
x&\mapsto P_{\mathcal D}f(x),
\end{align}
where $P_{\mathcal D}$ is the orthogonal projection onto $\mathsf H(\mathcal D)$. Then $\mathsf H(f)$ is well-defined and compatible with composition, in the sense that $\mathsf H(g)\circ\mathsf H(f)=\mathsf H(g\circ f)$.
\end{proposition}

\begin{proof}
The proof is identical to that of~\cite[Lemma 3.4]{Liu}, given there in the context of differential graded inner-product spaces.
\end{proof}

\begin{remark}
Proposition~\ref{prop:h-functor} says that $\mathsf H$ is a functor from the category of chain complexes and chain maps to the category of vector spaces and linear maps. In~\cite[Proposition 3.5]{Liu}, this functor is shown to be naturally isomorphic to the usual cohomology functor, giving a direct connection between maps induced on cohomology and maps induced on harmonic cochains.
\end{remark}

\section{Khovanov homology}\label{sec:khovanov-background}

In this section we recall Khovanov's categorification of the Jones polynomial, fix conventions used throughout the paper, and discuss the canonical inner product that equips each Khovanov cochain space with the structure of a Hilbert space. Our exposition follows Bar-Natan~\cite{Bar_Natan_2005} (see also Khovanov's original paper~\cite{Khovanov2000}). Although Khovanov's original construction is naturally a cohomology theory, it is conventional in the literature to refer to it as Khovanov \emph{homology}, and we follow this usage.

\subsection{The cube of resolutions}\label{ssec:kauffman}

Let $D$ be a planar diagram of an oriented link $K$, with $N$ ordered crossings. The Kauffman bracket replaces each crossing locally, in one of two ways, by a pair of non-intersecting arcs:
\begin{equation}\label{eq:smoothings}
\hackcenter{\begin{tikzpicture}[scale=0.55]
\draw[very thick] (0,0) -- (1,1);
\draw[very thick, white, line width=5pt] (1,0) -- (0,1);
\draw[very thick] (1,0) -- (0,1);
\node at (0.5,-0.55) {\scriptsize crossing};
\end{tikzpicture}}
\;\;\rightsquigarrow\;\;
\hackcenter{\begin{tikzpicture}[scale=0.55]
\draw[very thick] (0,0) .. controls (0.5,0.45) .. (1,0);
\draw[very thick] (0,1) .. controls (0.5,0.55) .. (1,1);
\node at (0.5,-0.55) {\scriptsize $0$-smoothing};
\end{tikzpicture}}
\qquad\text{or}\qquad
\hackcenter{\begin{tikzpicture}[scale=0.55]
\draw[very thick] (0,0) .. controls (0.45,0.5) .. (0,1);
\draw[very thick] (1,0) .. controls (0.55,0.5) .. (1,1);
\node at (0.5,-0.55) {\scriptsize $1$-smoothing};
\end{tikzpicture}}
\end{equation}
Throughout the paper we use the convention shown in~\eqref{eq:smoothings}, with the corresponding cube orientation and signs fixed in Section~\ref{ssec:khovanov-complex}, following~\cite{Bar_Natan_2005}.

A choice of smoothing at every crossing is recorded by a vector $r\in\{0,1\}^N$ called a \emph{state}, and the resulting crossingless diagram $D_r$ is a disjoint union of unknotted circles in the plane. We write
\[
\ell(r):=\text{the number of circles in the resolution }D_r.
\]
The $2^N$ states form the vertices of the Boolean cube $\{0,1\}^N$, oriented along edges by increasing Hamming weight $|r|=\sum_i r_i$. Flipping a single bit $r\rightsquigarrow r'$ from $0$ to $1$ produces $D_{r'}$ from $D_r$ either by fusing two circles into one, in which case $\ell(r')=\ell(r)-1$, or by splitting one circle into two, in which case $\ell(r')=\ell(r)+1$. See~\cite[Figure 2]{schmidhuber2025quantumalgorithmkhovanovhomology} for a worked example on the Hopf link.

\subsection{The Frobenius algebra \texorpdfstring{$V$}{V}}\label{ssec:frobenius}

Khovanov's cochain complex is built by replacing each circle in a resolved diagram $D_r$ by a copy of the two-dimensional graded vector space
\begin{equation}\label{eq:V-defn}
V=\C[x]/(x^2)\cong\C\langle 1,x\rangle,
\end{equation}
with $q$-grading $\deg(1)=1$ and $\deg(x)=-1$. As an ungraded algebra, $V$ is isomorphic to $\mathrm H^*(\mathbb{CP}^1;\C)$; the displayed $q$-grading is the conventional Khovanov regrading. The vector space $V$ carries the structure of a graded commutative Frobenius algebra. Its multiplication $m\colon V\otimes V\to V$ is given by
\begin{equation}\label{eq:multiplication}
m(1\otimes 1)=1,\qquad m(1\otimes x)=m(x\otimes 1)=x,\qquad m(x\otimes x)=0,
\end{equation}
and its comultiplication $\Delta_V\colon V\to V\otimes V$ is given by
\begin{equation}\label{eq:comultiplication}
\Delta_V(1)=1\otimes x+x\otimes 1,\qquad \Delta_V(x)=x\otimes x.
\end{equation}
The unit is $1\in V$, and the counit is the linear functional $\varepsilon\colon V\to\C$ with $\varepsilon(1)=0$ and $\varepsilon(x)=1$.

\subsection{The canonical inner product on Khovanov cochain spaces}\label{ssec:inner-product}

The Frobenius algebra structure on $V$ supplies a canonical symmetric bilinear pairing, the \emph{Frobenius pairing},
\begin{equation}\label{eq:frobenius-pairing}
B\colon V\otimes V\longrightarrow\C,\qquad B(a,b)=\varepsilon(a\cdot b).
\end{equation}
This is the Poincar\'e pairing under the ungraded algebra identification $V\cong\mathrm H^*(\mathbb{CP}^1;\C)$. A direct calculation gives the matrix of $B$ in the basis $\{1,x\}$,
\[
\bigl(B(e_i,e_j)\bigr)_{i,j}
=
\begin{pmatrix}
\varepsilon(1\cdot 1)&\varepsilon(1\cdot x)\\
\varepsilon(x\cdot 1)&\varepsilon(x\cdot x)
\end{pmatrix}
=
\begin{pmatrix}0&1\\1&0\end{pmatrix},
\]
with eigenvalues $\pm1$. Although nondegenerate, the Frobenius pairing is \emph{indefinite} and does not equip $V$ with the structure of a complex Hilbert space. Using the Frobenius pairing $B$ as an inner product on Khovanov cochain spaces leads to a Krein space in which the Hodge decomposition of Theorem~\ref{thm_Hodge} fails and the Laplacian is not positive semidefinite.

For the spectral analysis carried out in this paper, we instead use the positive-definite Hermitian inner product on $V$ for which $\{1,x\}$ is an orthonormal basis,
\begin{equation}\label{eq:hermitian-pairing}
\langle 1,1\rangle=\langle x,x\rangle=1,
\qquad
\langle 1,x\rangle=\langle x,1\rangle=0.
\end{equation}
This Hermitian metric is determined by the distinguished integral basis $\{1,x\}$ and is different from the Frobenius pairing. It extends multiplicatively to all tensor powers $V^{\otimes k}$. For any state $r$, the circle space $V^{\otimes\ell(r)}$ thus inherits a canonical Hilbert-space structure, and so does the total Khovanov cochain space.

\subsection{The Khovanov complex and its Laplacian}\label{ssec:khovanov-complex}

To each state $r\in\{0,1\}^N$, Khovanov associates the graded vector space $V^{\otimes\ell(r)}$, with an overall $q$-degree shift recording the Hamming weight $|r|$. Before the writhe normalization, the total cube-graded cochain space is
\begin{equation}\label{eq:khovanov-cochain}
\widetilde C(D)
=
\bigoplus_{r\in\{0,1\}^N}V^{\otimes\ell(r)}\{|r|\}.
\end{equation}
We write $\widetilde C^{k,q}(D)$ for its summand in cube degree $k=|r|$ and unnormalized $q$-degree $q$. The differential $\widetilde\partial$ is built from the edges of the cube. Each edge $r\to r'$ acts by the multiplication $m$ on the relevant tensor factors when $D_r\to D_{r'}$ fuses two circles, and by the comultiplication $\Delta_V$ when it splits one circle, with the identity acting on all remaining tensor factors. With the standard sign conventions of~\cite[\S 2]{Bar_Natan_2005}, $\widetilde\partial^2=0$.

Let $n_+(D)$ and $n_-(D)$ denote the numbers of positive and negative crossings of $D$. The normalized Khovanov bidegrees $(i,j)$ are related to the cube bidegrees $(k,q)$ by
\begin{equation}\label{eq:grading-translation}
i=k-n_-(D),
\qquad
j=q+n_+(D)-2n_-(D).
\end{equation}
We denote the resulting normalized cochain complex over a base ring $R$ by $C_{\mathrm{KH}}^{i,j}(D;R)$ and its differential by $\partial$. Its cohomology
\begin{equation}\label{eq:khovanov-cohomology}
\mathrm{KH}^{i,j}(K;R)
=
\mathrm H^i\bigl(C_{\mathrm{KH}}^{*,j}(D;R),\partial\bigr)
\end{equation}
is the link-invariant Khovanov homology~\cite{Khovanov2000}. We omit $R$ from the notation when it is clear from context and work over $R=\C$ unless stated otherwise. The Hermitian inner product~\eqref{eq:hermitian-pairing} extends multiplicatively to both the cube-graded and normalized cochain spaces after extension of scalars to $\C$.

\begin{definition}\label{khovanovlaplacian}
The normalized Khovanov Laplacian in bidegree $(i,j)$ is
\begin{equation}\label{eq:khovanov-laplacian}
\Delta_{i,j}^{\mathrm{KH}}(D)
:=
\partial_{i,j}^{\dagger}\partial_{i,j}
+
\partial_{i-1,j}\partial_{i-1,j}^{\dagger}.
\end{equation}
We write $\Delta_i^{\mathrm{KH}}(D)=\bigoplus_j\Delta_{i,j}^{\mathrm{KH}}(D)$. Likewise, $\widetilde\Delta_{k,q}(D)$ denotes the Laplacian of the unnormalized cube complex in bidegree $(k,q)$. Under the canonical identification supplied by~\eqref{eq:grading-translation},
\begin{equation}\label{eq:laplacian-shift}
\widetilde\Delta_{k,q}(D)
=
\Delta_{k-n_-(D),\,q+n_+(D)-2n_-(D)}^{\mathrm{KH}}(D)
\end{equation}
as matrices.
\end{definition}

\begin{remark}\label{rmk:hodge-khovanov}
Jones and Wei introduced the operator~\eqref{eq:khovanov-laplacian} and studied its spectrum numerically for small knots~\cite{Jones2025}. Each Khovanov Laplacian is self-adjoint and positive semidefinite. Combinatorial Hodge theory gives canonical isomorphisms
\[
\ker\bigl(\Delta_i^{\mathrm{KH}}(D)\bigr)\cong\mathrm{KH}^i(K;\C),
\qquad
\ker\bigl(\Delta_{i,j}^{\mathrm{KH}}(D)\bigr)\cong\mathrm{KH}^{i,j}(K;\C).
\]
The kernel therefore depends only on the underlying link $K$, not on the choice of diagram $D$. By contrast, the higher spectrum depends on $D$ and is not a link invariant. Identifying the topological information that the higher spectrum nonetheless carries is the central concern of the following sections. For example, the crossingless unknot has zero differential and hence no nonzero Laplacian spectrum, whereas a one-crossing Reidemeister~I diagram of the unknot has nonzero Laplacian eigenvalues $1$ and $2$; see Example~\ref{exam:twisted}.

The higher spectrum is not functorial under Reidemeister chain homotopy equivalences. A chain homotopy equivalence preserves cohomology, but it need not preserve adjoints, orthogonal complements, or Laplacian eigenvectors. In particular, one should not expect a Reidemeister move to induce monotonicity of the nonzero spectrum or of the spectral gap. Even when an adjoint-compatible map exists in a local summand, this property need not be preserved after the map is inserted into the tensor-product structure of the full Khovanov cube.
\end{remark}

\section{Spectral gaps via signless graph Laplacians}\label{sec:graph}

Since the higher spectrum of the Khovanov Laplacian is diagram-dependent and is not controlled by Reidemeister chain maps, quantitative statements require fixed diagrammatic models. In this section we identify the Khovanov Laplacian in the minimal and maximal cube degrees with signless graph Laplacians attached to the diagram, up to diagonal sign conjugation in maximal cube degree, and use this identification to derive both structural and quantitative consequences.

The structural results, collected in Section~\ref{ssec:structural}, identify rational Khovanov cohomology classes in the shifted extreme bidegrees with loop-free bipartite components of the associated graph (Theorem~\ref{thm:khovanov-bipartite}), produce explicit harmonic representatives via $\pm1$ vertex two-colorings (Corollary~\ref{cor:harmonic-2-coloring}), and extend from the all-zero to the all-one resolution after accounting for the cube signs by diagonal conjugation (Proposition~\ref{prop:max-degree}).

The quantitative side, developed in Sections~\ref{ssec:weighted-graph-bounds}--\ref{ssec:psi-bound}, recalls the bipartiteness-defect bound of Desai and Rao~\cite{Desai}, constructs the relevant graph $G_{n,k}(D)$ from a link diagram, and gives a fixed-width inverse-polynomial lower bound near the extremal $q$-degrees when the corresponding rational Khovanov cohomology vanishes (Proposition~\ref{prop:gap-when-no-Q-cohomology}). Sections~\ref{ssec:twisted-unknot-graphs}--\ref{ssec:even-twist-knots} compute the graphs, and the resulting bidegree gap formulas, for twisted unknots $TU_N$, odd $(2,N)$-torus knots $T_{(2,N)}$, and even twist knots $\mathrm{Tw}_N$.

% - - - - - - - - - - - - - - - - - -
\subsection{Structural consequences: bipartiteness, harmonic representatives, and duality}\label{ssec:structural}
% - - - - - - - - - - - - - - - - - -

Throughout this section, the \emph{signless graph Laplacian} of a finite weighted graph $G$ with weighted adjacency matrix $A_G$ and weighted degree matrix $D_G$ is
\begin{equation}\label{eq:signless-laplacian-defn}
Q(G):=D_G+A_G,
\end{equation}
in contrast to the standard graph Laplacian $L(G)=D_G-A_G$. Loops and positive weights are permitted. The precise conventions, together with the bipartiteness-defect bound that controls $\lambda_{\min}(Q(G))$, are recalled in Section~\ref{ssec:weighted-graph-bounds}.

The construction of Section~\ref{ssec:zero-degree-graph} assigns to each diagram $D$ and each unnormalized $q$-degree $q$ a finite weighted graph $G_{0,q}(D)$ such that the Khovanov Laplacian in cube degree zero satisfies
\begin{equation}\label{eq:khovanov-Q-identification}
\widetilde\Delta_{0,q}(D)=Q\bigl(G_{0,q}(D)\bigr).
\end{equation}
The vertex set of $G_{0,q}(D)$ indexes the standard basis of $V^{\otimes\ell(\mathbf 0)}$ in unnormalized $q$-degree $q$, where $\ell(\mathbf 0)$ is the number of circles in the all-zero resolution. Edges correspond to multiplication crossings, and loops correspond to comultiplication crossings. We record three structural consequences of~\eqref{eq:khovanov-Q-identification} before turning to the detailed construction.

\subsubsection{Bipartiteness corresponds to rational cohomology}

For any weighted graph $G$ that may contain loops,
\begin{equation}\label{eq:Q-kernel-bipartite}
\dim\ker Q(G)
=
\#\bigl\{\text{loop-free bipartite connected components of }G\bigr\}.
\end{equation}
This is standard in spectral graph theory~\cite{HAEMERS2004199,Desai}. A loop counts as an odd cycle of length one, so any component containing a loop or an odd cycle contributes nothing to $\ker Q(G)$. Each loop-free bipartite component contributes a one-dimensional kernel, spanned by the $\pm1$ vertex labeling obtained from a two-coloring.

\begin{thm}\label{thm:khovanov-bipartite}
Let $D$ be a diagram of an oriented link $K$, and let $q\in\Z$. Then
\begin{equation}\label{eq:khovanov-bipartite}
\dim_{\Q}\mathrm{KH}^{-n_-(D),\,q+n_+(D)-2n_-(D)}(K;\Q)
=
\#\bigl\{\text{loop-free bipartite connected components of }G_{0,q}(D)\bigr\}.
\end{equation}
In particular, the Khovanov group on the left vanishes if and only if every connected component of $G_{0,q}(D)$ contains a loop or an odd cycle.
\end{thm}

\begin{proof}
By Hodge theory and the grading translation~\eqref{eq:laplacian-shift},
\[
\dim_{\C}\ker\widetilde\Delta_{0,q}(D)
=
\dim_{\C}\mathrm{KH}^{-n_-(D),\,q+n_+(D)-2n_-(D)}(K;\C)
=
\dim_{\Q}\mathrm{KH}^{-n_-(D),\,q+n_+(D)-2n_-(D)}(K;\Q).
\]
By~\eqref{eq:khovanov-Q-identification}, the left-hand side equals $\dim_{\C}\ker Q(G_{0,q}(D))$, which is the number of loop-free bipartite connected components by~\eqref{eq:Q-kernel-bipartite}.
\end{proof}

\subsubsection{Harmonic representatives via two-colorings}

The proof of Theorem~\ref{thm:khovanov-bipartite} produces explicit harmonic representatives.

\begin{cor}\label{cor:harmonic-2-coloring}
For each loop-free bipartite connected component $C\subseteq G_{0,q}(D)$ with bipartition $V(C)=A\sqcup B$, define
\[
h_C
=
\sum_{v\in A}e_v-
\sum_{v\in B}e_v
\in\widetilde C^{0,q}(D).
\]
Then $h_C$ is harmonic, and the vectors $h_C$, indexed by the loop-free bipartite connected components of $G_{0,q}(D)$, form a basis of
\[
\ker\widetilde\Delta_{0,q}(D).
\]
\end{cor}

\subsubsection{Duality and maximal cube degree}

The all-zero and all-one resolutions of $D$ are exchanged by reversing the orientation of every edge of the cube of resolutions. Under this duality, multiplications and comultiplications are interchanged via Frobenius duality on $V$. The only additional bookkeeping comes from the cube signs. In maximal cube degree, these signs produce a signed version of the graph matrix, but they are removed by a diagonal change of basis: if an off-diagonal entry arises with signs $\epsilon_u$ and $\epsilon_v$ at its endpoints, then its sign is $\epsilon_u\epsilon_v$.

\begin{proposition}\label{prop:max-degree}
Let $D$ have $N$ crossings. For every unnormalized $q$-degree $q$, there is a finite weighted graph $G_{N,q}(D)$ obtained from the all-one resolution and a diagonal sign matrix $S$ with entries $\pm1$ such that
\[
S\widetilde\Delta_{N,q}(D)S^{-1}=Q\bigl(G_{N,q}(D)\bigr).
\]
Consequently,
\[
\dim_{\Q}\mathrm{KH}^{N-n_-(D),\,q+n_+(D)-2n_-(D)}(K;\Q)
=
\#\bigl\{\text{loop-free bipartite connected components of }G_{N,q}(D)\bigr\},
\]
and the harmonic-representative description of Corollary~\ref{cor:harmonic-2-coloring} holds in maximal cube degree after conjugating back by $S$.
\end{proposition}

\subsubsection{A fixed-width gap bound when rational cohomology vanishes}

Theorem~\ref{thm:khovanov-bipartite} converts the absence of rational Khovanov cohomology into a graph-theoretic statement that controls the spectral gap.

\begin{proposition}\label{prop:gap-when-no-Q-cohomology}
Suppose $\ker\widetilde\Delta_{0,q}(D)=0$. Write $|V(G_{0,q}(D))|$ for the number of vertices of $G_{0,q}(D)$, and let $d^*$ be its maximum weighted vertex degree. Then the bipartiteness-defect parameter $\Psi$ satisfies
\[
\Psi\geq\frac{1}{2|V(G_{0,q}(D))|},
\]
and
\begin{equation}\label{eq:polynomial-gap-bound}
\operatorname{gap}\bigl(\widetilde\Delta_{0,q}(D)\bigr)
\geq
\frac{1}{16d^*|V(G_{0,q}(D))|^2}.
\end{equation}
If the all-zero resolution has $\ell$ circles and $q=k-n$ with $n+k=\ell$, then
\[
|V(G_{0,q}(D))|=\binom{\ell}{k}.
\]
In particular, for a family of $N$-crossing diagrams with $\ell=O(N)$, Equation~\eqref{eq:polynomial-gap-bound} is inverse polynomial in every fixed-width band near either extremal $q$-degree.
\end{proposition}

\begin{proof}
By Theorem~\ref{thm:khovanov-bipartite}, every connected component of $G_{0,q}(D)$ contains a loop or an odd cycle. Fix a nonempty subset $S\subseteq V(G_{0,q}(D))$. If the induced subgraph $G[S]$ is non-bipartite, then $e_{\min}(S)\geq 1/2$, since the smallest positive loop or edge weight in the Khovanov graphs is $1/2$. If $G[S]$ is bipartite and $\operatorname{cut}(S)=0$, then $S$ is a union of loop-free bipartite connected components, contradicting the hypothesis. Hence $\operatorname{cut}(S)\geq1$ in the second case. In either case,
\[
\Psi_S\geq\frac{1}{2|S|}\geq\frac{1}{2|V(G_{0,q}(D))|}.
\]
The lower Desai--Rao bound gives~\eqref{eq:polynomial-gap-bound}. The vertex count follows directly from choosing which $k$ of the $\ell$ tensor factors are labeled by $1$. Finally, $d^*\leq N$, since each crossing contributes at most one to the weighted degree of a vertex.
\end{proof}

\begin{remark}\label{rmk:fixed-width-only}
The fixed-width hypothesis is essential for this argument. If $k$ is fixed, then $\binom{\ell}{k}=O(\ell^k)$, and the same holds near the opposite extremal degree with $n$ fixed. In a middle $q$-degree, however, $k$ can be of order $\ell/2$, and
\[
\binom{\ell}{\lfloor\ell/2\rfloor}
\sim
\sqrt{\frac{2}{\pi\ell}}\,2^\ell.
\]
Thus the Desai--Rao estimate alone gives no inverse-polynomial lower bound in the bulk of the $q$-grading. Extending the result there would require additional structural control of the bipartiteness parameter $\Psi$, not merely a sharper estimate of the total vertex count.
\end{remark}

\subsubsection{Open question: graph-level behavior under local diagram changes}\label{sec:graph functoriality of Reidemeister moves}

It is natural to ask how local modifications of a diagram are reflected in the extreme-degree graphs $G_{0,q}(D)$ and $G_{N,q}(D)$. In favorable cases these changes may correspond to explicit graph operations such as vertex insertions, edge insertions, contractions, or changes in loop weights. Identifying these operations, and understanding how the signless-Laplacian spectrum behaves under them, would give a diagrammatic explanation for changes in the Khovanov spectral gap. We do not expect such a description, by itself, to imply monotonicity of spectral gaps under Reidemeister moves. We leave this as an open problem.

% - - - - - - - - - - - - - - - - - -
\subsection{Bounds on minimal eigenvalues of weighted graphs}\label{ssec:weighted-graph-bounds}
% - - - - - - - - - - - - - - - - - -

We recall the bipartiteness-defect bound of Desai and Rao~\cite{Desai} for the smallest eigenvalue of a signless graph Laplacian, in the form used below.

Let $G=(V,E)$ be a finite weighted graph with positive edge weights and possible self-loops. Write $A_G$ for its weighted adjacency matrix and $D_G$ for its diagonal matrix of weighted degrees. A loop of weight $w$ contributes $w$ to both $(A_G)_{vv}$ and $(D_G)_{vv}$. Thus
\[
Q(G)=D_G+A_G.
\]
For a vertex subset $S\subseteq V$, write $G[S]$ for the induced subgraph, let $e_{\min}(S)$ be the minimum total weight of edges and loops that must be removed from $G[S]$ to make it bipartite, and let $\operatorname{cut}(S)$ be the total weight of edges crossing from $S$ to $V\setminus S$. Set
\begin{equation}\label{eq:defPsi}
\Psi(G):=\min_{\varnothing\neq S\subseteq V}\Psi_S,
\qquad
\Psi_S:=\frac{e_{\min}(S)+\operatorname{cut}(S)}{|S|}.
\end{equation}
A graph has $\Psi(G)=0$ if and only if it has a loop-free bipartite connected component. We write $d^*(G)$ for the largest weighted vertex degree.

The following results were proved in~\cite{Desai} for unweighted graphs. The same Rayleigh-quotient argument gives the weighted form used here.

\begin{thm}[Desai--Rao]\label{thm:Desai}
The matrix $Q(G)$ is positive semidefinite, and
\[
\ker Q(G)\neq0
\quad\Longleftrightarrow\quad
G\text{ has a loop-free bipartite connected component}.
\]
Moreover,
\begin{equation}\label{eq:desai-rao-bound}
\frac{\Psi(G)^2}{4d^*(G)}
\leq
\lambda_{\min}(Q(G))
\leq
4\Psi(G).
\end{equation}
\end{thm}

\subsection{Graphs from Khovanov Laplacians in cube degree zero}\label{ssec:zero-degree-graph}
% - - - - - - - - - - - - - - - - - -

% - - - - - - - - - - - - - - - - - -
\subsubsection{The structure of the Khovanov complex in cube degree zero} \label{subsubsec:homzero-structure}
% - - - - - - - - - - - - - - - - - -
The signless graph model is cleanest in cube degree zero because there is no incoming differential. In the all-zero resolution, the cube signs on outgoing edges are all $+1$, and the local summands $m^\dagger m$ and $\Delta_V^\dagger\Delta_V$ have nonnegative matrix entries in the standard basis. Thus the degree-zero Laplacian is represented by a genuine signless graph Laplacian with positive weights. In intermediate cube degrees, incoming and outgoing cube signs interact and one generally obtains signed matrices rather than ordinary signless graph Laplacians. The maximal-cube-degree case is recovered separately by duality, with the signs removed by diagonal conjugation.

To describe the graph $G_{n,k}(D)$ of Section~\ref{ssec:zero-degree-graph} below, we first fix a convenient indexing of the tensor factors of $\widetilde C^0(D)$. Throughout, $\ket{r} = \ket{0 \cdots 0}$ denotes the all-zero resolution of a diagram $D$ with $N$ crossings, and $\ell = \ell(r)$ the number of resolution circles. By Section~\ref{ssec:khovanov-complex}, $\widetilde C^0(D) \cong V^{\otimes \ell}$, with one copy of $V=\C\langle1,x\rangle$ for each circle.

The tensor factors are indexed by referring to the underlying planar diagram. Treat each crossing as a $4$-valent vertex; the resulting plane graph has $2N$ edges connecting its $N$ vertices. Choose a basepoint on $K$ and label the edges $1, 2, \ldots, 2N$ in the order encountered along a complete traversal of $K$. After resolving the diagram at every crossing, each resulting circle inherits two or more of these edge labels, and we identify the corresponding tensor factor of $V$ by the smallest such label. With this convention,
\[
V^{\otimes \ell} \;=\; V_{i_1} \otimes V_{i_2} \otimes \cdots \otimes V_{i_\ell}
\]
where $1 \le i_1 < i_2 < \cdots < i_\ell \le 2N$ are the chosen indices.

\begin{example} \label{example:trefoil-label}
Consider the standard diagram $D_{6_3}$ of the knot $6_3$ with a labeling $\und{1}, \dots, \und{6}$ of its crossings and a labeling from $1$ to $12$ of its edges.  Below we illustrate the process of labeling all edge components and how this labeling gives rise to an indexing of the tensor powers appearing in the all-zero resolution of this knot.
\[
\xy
(0,0)*+{\includegraphics[width=1.8in]{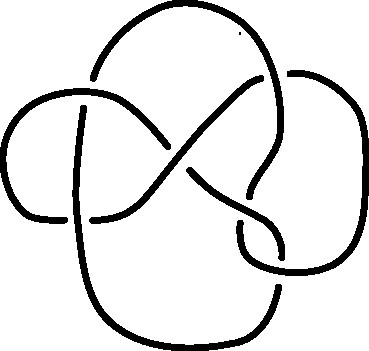}}:
(0,0)*+{
    \begin{tikzpicture}[scale=1]
    %
     % %   
        % Orange numbers
        \node[black] at (1.05, 2.9) {$\scriptstyle 1$};
        \node[black] at (.1, 1.16) {$\scriptstyle 2$};
        \node[black] at (-1.9, 2) {$\scriptstyle 3$};
        \node[black] at (.1, 2.9) {$\scriptstyle 4$};
        \node[black] at (.8, 1.35) {$\scriptstyle 5$};
         \node[black] at (2, 1.02) {$\scriptstyle 6$};
        \node[black] at (.63, -0.7) {$\scriptstyle 7$};
        \node[black] at (-.9, 2) {$\scriptstyle 8$};
        \node[black] at (.63, 4.2) {$\scriptstyle 9$};
        \node[black] at (2.1, 2.2) {$\scriptstyle 10$};
         \node[black] at (1.1, .8) {$\scriptstyle 11$};
        \node[black] at (3.1, 2) {$\scriptstyle 12$};
      %  
        % Red numbers with underlines
        \node[red] at (.61, 2.45) {\underline{1}};
        \node[red] at (-.7, 3.2) {\underline{2}};
        \node[red] at (2.1, 3.3) {\underline{3}};
        \node[red] at (1.3, 1.85) {\underline{4}};
        \node[red] at (2.15, .15) {\underline{5}}; 
        \node[red] at (-.9, 0.7) {\underline{6}};
    \end{tikzpicture}
    };
\endxy  
\qquad 
\quad 
    \xy
    (0,0)*+{\includegraphics[width=1.8in]{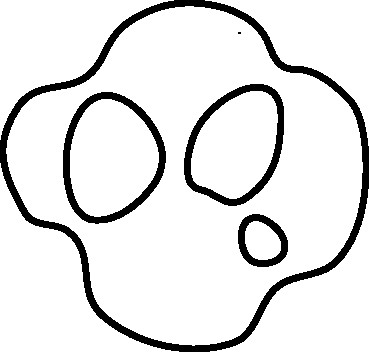}}:
    (0,0)*+{
        \begin{tikzpicture}[scale=1]
        %
         % %   
            % Orange numbers
            \node[black] at (1.05, 2.9) {$\scriptstyle 1$};
            \node[black] at (.1, 1.16) {$\scriptstyle 2$};
            \node[black] at (-1.9, 2) {$\scriptstyle 3$};
            \node[black] at (.1, 2.9) {$\scriptstyle 4$};
            \node[black] at (.8, 1.35) {$\scriptstyle 5$};
             \node[black] at (2, 1.02) {$\scriptstyle 6$};
            \node[black] at (.63, -0.6) {$\scriptstyle 7$};
            \node[black] at (-.9, 2) {$\scriptstyle 8$};
            \node[black] at (.63, 4.1) {$\scriptstyle 9$};
            \node[black] at (2.1, 2.2) {$\scriptstyle 10$};
             \node[black] at (1.1, .8) {$\scriptstyle 11$};
            \node[black] at (3.1, 2) {$\scriptstyle 12$};
        \end{tikzpicture}
        };
    \endxy  
\]
On the right, each loop is labeled by multiple edges from the original knot.  The tensor factors carry the minimal label appearing on each loop. In this case, $V_1 \otimes V_2 \otimes V_3 \otimes V_6$.
\end{example}

The cochain space $V^{\otimes \ell}$ decomposes by $q$-degree into components $V^{\otimes \ell}_q$, where the homogeneous basis vectors are tensor products $v = v_{i_1} \otimes \cdots \otimes v_{i_\ell}$ with $v_i \in \{1, x\}$, and the unnormalized $q$-degree is $q = k - n$, with $k$ the number of factors equal to $1$ and $n$ the number equal to $x$.

For a diagram $D$ with $N$ crossings, the cube-degree-zero piece $\widetilde\partial_0$ of the differential decomposes as a sum of $N$ pieces $d_{rr'}$, one for each resolution $\ket{r'}$ obtained by flipping a single $0$ to a $1$ in $\ket{r}$ (see Section~\ref{ssec:khovanov-complex}). Each $d_{rr'}$ is one of two local moves on the resolution circles: a multiplication $m_{ij} \colon V^{\otimes \ell} \to V^{\otimes (\ell - 1)}$ (when two circles fuse into one) or a comultiplication $s_{i,j} \colon V^{\otimes \ell} \to V^{\otimes (\ell + 1)}$ (when one circle splits into two), acting as the identity on all unindexed tensor factors. Their actions on the relevant factors are
\begin{alignat*}{3}
 m_{ij} \maps V_i\otimes V_j&\to V_{\min(i,j)}   \qquad \qquad    && s_{i,j} \maps V_{\min(i,j)} \to V_i \otimes V_j\\
 \ket{11} &\mapsto \ket{1} \qquad \qquad    && \quad\; \ket{1} \mapsto \ket{1x} + \ket{x1} \\
 \ket{1x} &\mapsto \ket{x} \qquad \qquad    && \quad\; \ket{x} \mapsto \ket{xx}\\
\ket{x1}  &\mapsto \ket{x} \\
\ket{xx}  &\mapsto 0 
\end{alignat*}
where $i$ and $j$ are indices of two tensor factors in $V^{\otimes \ell}$ and $\min(i,j)$ is a label of a single tensor factor in $V^{\otimes \ell}$.  In $s_{i,j}$, the newly created circle is labeled by $\max(i,j)$. 

The Laplacian $\widetilde\Delta_{0}(D)\maps V^{\otimes \ell} \to V^{\otimes \ell}$ in cube degree zero is then a sum of $N$ local operators of the form 
 \begin{alignat}{3} \label{eq:Delta0-bar}
 \mathfrak{m}_{i,j}:=m^{\dagger}_{ij}m_{ij} \maps V_i \otimes V_j&\to V_i \otimes V_j   \qquad \qquad    
        && \mathfrak{s}_{i} := s^{\dagger}_{i,j}s_{i,j} \maps V_i \to V_i  \nonumber \\
 \ket{11} &\mapsto \ket{11} \qquad \qquad    
    && \quad\; \ket{1} \mapsto 2\ket{1 }   \nonumber\\
 \ket{1x} &\mapsto \ket{1x} + \ket{x1} \qquad \qquad    && \quad\; \ket{x} \mapsto \ket{x}\\
\ket{x1}  &\mapsto \ket{1x} + \ket{x1}   \nonumber\\
\ket{xx}  &\mapsto 0  \nonumber
\end{alignat}
 that act by the identity on the non-indexed tensor factors.  In particular,
\begin{equation} \label{eq:0Lsplit}
  \widetilde\Delta_0(D) = \sum_{a=1}^{p} \mathfrak{m}_{i_a, j_a} + \sum_{b=1}^t \mathfrak{s}_{i_b} 
\end{equation}
for some   $1 \leq i_a < j_a \leq 2N$, $1 \leq i_b\leq 2N$, where $N=p+t$, $p$ is the number of merge terms, and $t$ is the number of split terms in $\widetilde\Delta_0(D)$.

\begin{example}\label{example:trefoil-zero}
For the standard diagram $D_{6_3}$, the all-zero resolution is a tensor product of $\ell=4$ copies of $V$ carrying labels $V_1 \otimes V_2 \otimes V_3 \otimes V_6$.  The Laplacian $\widetilde\Delta_0(D_{6_3})$ takes the form
\[ 
\widetilde\Delta_0(D_{6_3}) = \mathfrak{m}_{12}+2\mathfrak{m}_{23} + \mathfrak{m}_{13} + \mathfrak{m}_{16} + \mathfrak{m}_{36}
\]
Figure~\ref{fig:6-3-Hom=-zero} illustrates each term of the differential and the corresponding resolutions. 
\end{example}

\begin{figure}
    \centering
$
\xy
%% ALL 0 resolution
(-30,0)*+{
    \xy
    (0,0)*+{\includegraphics[width=1.35in]{6_3-000000.jpg}}:
    (0,0)*+{
        \begin{tikzpicture}[scale=.8]
        %
         % %   
            % Orange numbers
            \node[black] at (1.05, 2.9) {$\scriptstyle 1$};
            \node[black] at (.1, 1.16) {$\scriptstyle 2$};
            \node[black] at (-1.9, 2) {$\scriptstyle 3$};
            \node[black] at (.1, 2.9) {$\scriptstyle 4$};
            \node[black] at (.8, 1.35) {$\scriptstyle 5$};
             \node[black] at (2, 1.02) {$\scriptstyle 6$};
            \node[black] at (.63, -0.6) {$\scriptstyle 7$};
            \node[black] at (-.9, 2) {$\scriptstyle 8$};
            \node[black] at (.63, 4.1) {$\scriptstyle 9$};
            \node[black] at (2.1, 2.2) {$\scriptstyle 10$};
             \node[black] at (1.1, .8) {$\scriptstyle 11$};
            \node[black] at (3.1, 2) {$\scriptstyle 12$};
        \end{tikzpicture}
        };
    \endxy  
    }="0";
%
% 100000
(18,65)*{
    \xy
    (0,0)*+{\includegraphics[width=1in]{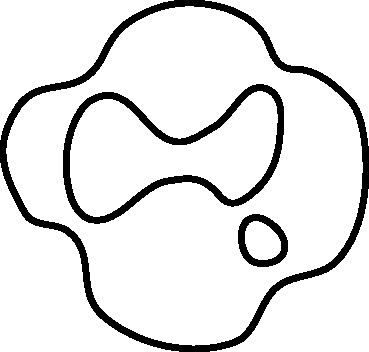}}:
    (0,0)*+{
    \begin{tikzpicture}[scale=.6]
    %
     % %   
        % Orange numbers
        \node[black] at (1.05, 2.9) {$\scriptscriptstyle 1$};
        \node[black] at (.1, 1.16) {$\scriptscriptstyle 2$};
        \node[black] at (-1.9, 2) {$\scriptscriptstyle 3$};
        \node[black] at (.1, 2.9) {$\scriptscriptstyle 4$};
        \node[black] at (.8, 1.35) {$\scriptscriptstyle 5$};
         \node[black] at (2, 1.02) {$\scriptscriptstyle 6$};
        \node[black] at (.63, -0.5) {$\scriptscriptstyle 7$};
        \node[black] at (-.9, 2) {$\scriptscriptstyle 8$};
        \node[black] at (.63, 4.0) {$\scriptscriptstyle 9$};
        \node[black] at (2.1, 2.2) {$\scriptscriptstyle 10$};
         \node[black] at (1.1, .8) {$\scriptscriptstyle 11$};
        \node[black] at (3.1, 2) {$\scriptscriptstyle 12$};
    \end{tikzpicture}
         };
    \endxy  
    }="t1";
%
% 010000
(32,40)*+{
    \xy
    (0,0)*+{\includegraphics[width=1in]{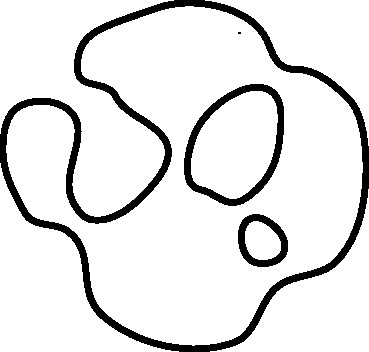}}:
    (0,0)*+{
    \begin{tikzpicture}[scale=.6]
    %
     % %   
        % Orange numbers
        \node[black] at (1.05, 2.9) {$\scriptscriptstyle 1$};
        \node[black] at (.1, 1.16) {$\scriptscriptstyle 2$};
        \node[black] at (-1.9, 2) {$\scriptscriptstyle 3$};
        \node[black] at (.1, 2.9) {$\scriptscriptstyle 4$};
        \node[black] at (.8, 1.35) {$\scriptscriptstyle 5$};
         \node[black] at (2, 1.02) {$\scriptscriptstyle 6$};
        \node[black] at (.63, -0.5) {$\scriptscriptstyle 7$};
        \node[black] at (-.9, 2) {$\scriptscriptstyle 8$};
        \node[black] at (.63, 4.0) {$\scriptscriptstyle 9$};
        \node[black] at (2.1, 2.2) {$\scriptscriptstyle 10$};
         \node[black] at (1.1, .8) {$\scriptscriptstyle 11$};
        \node[black] at (3.1, 2) {$\scriptscriptstyle 12$};
    \end{tikzpicture}
         };
    \endxy  
    }="t2";
%
% 001000
(45,15)*+{
    \xy
    (0,0)*+{\includegraphics[width=1in]{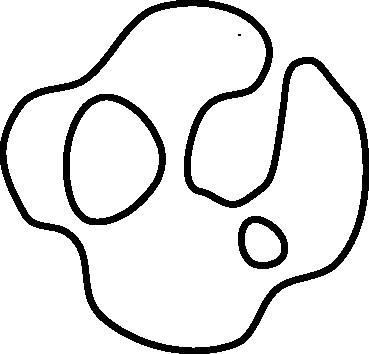}}:
    (0,0)*+{
    \begin{tikzpicture}[scale=.6]
    %
     % %   
        % Orange numbers
        \node[black] at (1.05, 2.9) {$\scriptscriptstyle 1$};
        \node[black] at (.1, 1.16) {$\scriptscriptstyle 2$};
        \node[black] at (-1.9, 2) {$\scriptscriptstyle 3$};
        \node[black] at (.1, 2.9) {$\scriptscriptstyle 4$};
        \node[black] at (.8, 1.35) {$\scriptscriptstyle 5$};
         \node[black] at (2, 1.02) {$\scriptscriptstyle 6$};
        \node[black] at (.63, -0.5) {$\scriptscriptstyle 7$};
        \node[black] at (-.9, 2) {$\scriptscriptstyle 8$};
        \node[black] at (.63, 4.0) {$\scriptscriptstyle 9$};
        \node[black] at (2.1, 2.2) {$\scriptscriptstyle 10$};
         \node[black] at (1.1, .8) {$\scriptscriptstyle 11$};
        \node[black] at (3.1, 2) {$\scriptscriptstyle 12$};
    \end{tikzpicture}
         };
    \endxy  
    }="t3";
%
% 000100
(45,-15)*+{
    \xy
    (0,0)*+{\includegraphics[width=1in]{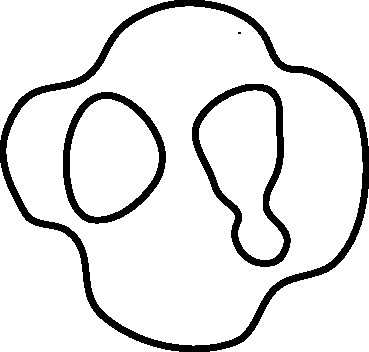}}:
    (0,0)*+{
    \begin{tikzpicture}[scale=.6]
    %
     % %   
        % Orange numbers
        \node[black] at (1.05, 2.9) {$\scriptscriptstyle 1$};
        \node[black] at (.1, 1.16) {$\scriptscriptstyle 2$};
        \node[black] at (-1.9, 2) {$\scriptscriptstyle 3$};
        \node[black] at (.1, 2.9) {$\scriptscriptstyle 4$};
        \node[black] at (.8, 1.35) {$\scriptscriptstyle 5$};
         \node[black] at (2, 1.02) {$\scriptscriptstyle 6$};
        \node[black] at (.63, -0.5) {$\scriptscriptstyle 7$};
        \node[black] at (-.9, 2) {$\scriptscriptstyle 8$};
        \node[black] at (.63, 4.0) {$\scriptscriptstyle 9$};
        \node[black] at (2.1, 2.2) {$\scriptscriptstyle 10$};
         \node[black] at (1.1, .8) {$\scriptscriptstyle 11$};
        \node[black] at (3.1, 2) {$\scriptscriptstyle 12$};
    \end{tikzpicture}
         };
    \endxy  
    }="t4";
%
% 000010
(32,-40)*+{
    \xy
    (0,0)*+{\includegraphics[width=1in]{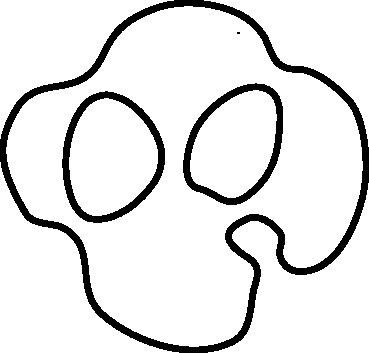}}:
    (0,0)*+{
    \begin{tikzpicture}[scale=.6]
    %
     % %   
        % Orange numbers
        \node[black] at (1.05, 2.9) {$\scriptscriptstyle 1$};
        \node[black] at (.1, 1.16) {$\scriptscriptstyle 2$};
        \node[black] at (-1.9, 2) {$\scriptscriptstyle 3$};
        \node[black] at (.1, 2.9) {$\scriptscriptstyle 4$};
        \node[black] at (.8, 1.35) {$\scriptscriptstyle 5$};
         \node[black] at (2, 1.02) {$\scriptscriptstyle 6$};
        \node[black] at (.63, -0.5) {$\scriptscriptstyle 7$};
        \node[black] at (-.9, 2) {$\scriptscriptstyle 8$};
        \node[black] at (.63, 4.0) {$\scriptscriptstyle 9$};
        \node[black] at (2.1, 2.2) {$\scriptscriptstyle 10$};
         \node[black] at (1.1, .8) {$\scriptscriptstyle 11$};
        \node[black] at (3.1, 2) {$\scriptscriptstyle 12$};
    \end{tikzpicture}
         };
    \endxy  
    }="t5";
%
% 000001
(18,-65)*+{
    \xy
    (0,0)*+{\includegraphics[width=1in]{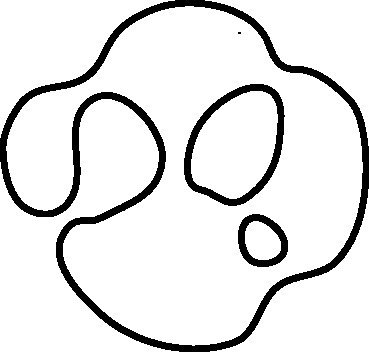}}:
    (0,0)*+{
    \begin{tikzpicture}[scale=.6]
    %
     % %   
        % Orange numbers
        \node[black] at (1.05, 2.9) {$\scriptscriptstyle 1$};
        \node[black] at (.1, 1.16) {$\scriptscriptstyle 2$};
        \node[black] at (-1.9, 2) {$\scriptscriptstyle 3$};
        \node[black] at (.1, 2.9) {$\scriptscriptstyle 4$};
        \node[black] at (.8, 1.35) {$\scriptscriptstyle 5$};
         \node[black] at (2, 1.02) {$\scriptscriptstyle 6$};
        \node[black] at (.63, -0.5) {$\scriptscriptstyle 7$};
        \node[black] at (-.9, 2) {$\scriptscriptstyle 8$};
        \node[black] at (.63, 4.0) {$\scriptscriptstyle 9$};
        \node[black] at (2.1, 2.2) {$\scriptscriptstyle 10$};
         \node[black] at (1.1, .8) {$\scriptscriptstyle 11$};
        \node[black] at (3.1, 2) {$\scriptscriptstyle 12$};
    \end{tikzpicture}
         };
    \endxy  
    }="t6";
{\ar^{d_{\ast00000}=m_{12}} "0"+(10,20);"t1"}; 
{\ar^{d_{0\ast0000}=m_{23}} "0";"t2"}; 
{\ar^{d_{00\ast000}=m_{13}} "0";"t3"}; 
{\ar^{d_{000\ast00}=m_{16}} "0";"t4"}; 
{\ar^{d_{0000\ast0}=m_{36}} "0";"t5"}; 
{\ar_{d_{00000\ast}=m_{23}} "0";"t6"}; 
(-48,15)*{\scriptstyle V_1\otimes V_2 \otimes V_3 \otimes V_6};
(34,75)*{\scriptstyle V_1 \otimes V_3 \otimes V_6};
(48,50)*{\scriptstyle V_1 \otimes V_2 \otimes V_6};
(62,25)*{\scriptstyle V_1 \otimes V_2 \otimes V_6};
(62,-05)*{\scriptstyle V_1 \otimes V_2 \otimes V_3};
(50,-33)*{\scriptstyle V_1 \otimes V_2 \otimes V_3};
(35,-57)*{\scriptstyle V_1 \otimes V_2 \otimes V_6};
\endxy 
$
    \caption{An illustration of the cochain spaces of the knot $6_3$ in cube degrees zero and one.  For this knot, all components of the differential $\widetilde\partial\maps\widetilde C^0\to\widetilde C^1$ come from merge or multiplication maps of the various loops.   }
    \label{fig:6-3-Hom=-zero}
\end{figure}

The cube-degree-zero Laplacian splits as a direct sum over unnormalized $q$-degree:
\[
\widetilde\Delta_0(D)
=
\bigoplus_{k=0}^{\ell}\widetilde\Delta_{0,2k-\ell}(D),
\]
where $k$ is the number of factors labeled by $1$, $n=\ell-k$ is the number labeled by $x$, and $q=k-n=2k-\ell$.

% - - - - - - - - - - - - - - - - - -
\subsubsection{Defining the cube-degree-zero graphs $G_{n,k}(D)$}
% - - - - - - - - - - - - - - - - - -

We now use the decomposition~\eqref{eq:0Lsplit} of $\widetilde\Delta_0(D)$ to construct, for each pair $(n, k)$ with $n + k = \ell$, a weighted graph $G_{n,k}(D)$ on $\binom{n+k}{k}$ vertices whose signless Laplacian recovers $\widetilde\Delta_{0,k-n}(D)$. The graph encodes the action of the multiplications $\mathfrak{m}_{ij}$ and split operators $\mathfrak{s}_i$ on the basis vectors.

\begin{itemize}
    \item \emph{Vertices.} Take $V$ to be the set of sequences $v = v_{i_1} v_{i_2} \cdots v_{i_\ell}$ of length $\ell$ with $n$ entries equal to $x$ and $k$ entries equal to $1$.  
    The indices of $v$ are chosen to match the circle labels of the all-zero resolution of $D$ as explained in Section~\ref{subsubsec:homzero-structure}.  
    
    \item For each $\mathfrak{m}_{ij}$ appearing in Equation~\eqref{eq:0Lsplit},  
        \begin{itemize}
            \item we add a loop from $v$ to itself with weight $w_{ii}=\frac{1}{2}$ if $v_i=v_j=1$,
            \item we add an edge from $v=v_1 \dots v_i \dots v_j \dots v_\ell$ to $v'=v_1 \dots v_j \dots v_i \dots v_\ell$, where the $i$th and $j$th terms have been swapped, if $v_iv_j \in \{1x, x1\}$. 
        \end{itemize} 
    \item For each $\mathfrak{s}_{i}$ appearing in Equation~\eqref{eq:0Lsplit}, 
        \begin{itemize}
            \item we add a loop of weight $\frac{1}{2}$ to vertex $v$ if $v_i = x$,
            \item add a loop of weight $+1$ to vertex $v$ if $v_i = 1$. 
        \end{itemize}
\end{itemize}

Note that because $Q(G)=D_G+A_G$ and the matrices $A_G$ and $D_G$ both have the edge weights associated with loops, the matrix $Q$ will always be integer valued with no fractional values. 

\begin{example} \label{example:6-3-graph}
Let $D_{6_3}$ be the standard diagram with crossings and edges labeled as in Example~\ref{example:trefoil-label}.  As shown in Example~\ref{example:trefoil-zero}, the all-zero resolution is a tensor product of $\ell=4$ copies of $V$ carrying labels $V_1 \otimes V_2 \otimes V_3 \otimes V_6$.  The Laplacian $\widetilde\Delta_0(D_{6_3})$ takes the form
\[ 
\widetilde\Delta_0(D_{6_3}) = \mathfrak{m}_{12}+2\mathfrak{m}_{23} + \mathfrak{m}_{13} + \mathfrak{m}_{16} + \mathfrak{m}_{36}.
\]

Consider $n=k=2$ corresponding to $q$-degree zero.  The Laplacian in this cube bidegree takes the form 
\[
\widetilde\Delta_{0,0}(D_{6_3}) =  
\begin{pmatrix}
5 & 2 & 1 & 0 & 1 & 0 \\
2 & 6 & 1 & 1 & 0 & 1 \\
1 & 1 & 5 & 0 & 1 & 0 \\
0 & 1 & 0 & 4 & 1 & 1 \\
1 & 0 & 1 & 1 & 5 & 2 \\
0 & 1 & 0 & 1 & 2 & 5 \\
\end{pmatrix}
\]
where we have chosen the ordered basis $\{11xx, 1x1x, x11x, 1xx1, x1x1, xx11\}$ to match conventions in Bar-Natan's code for computing Khovanov homology and it differs from our choices in Section~\ref{sec:torsion-knots}.  
The graph $G_{2,2}(D_{6_3})$ then takes the form
\[
\xy
(0,0)*+{\includegraphics[width=5in]{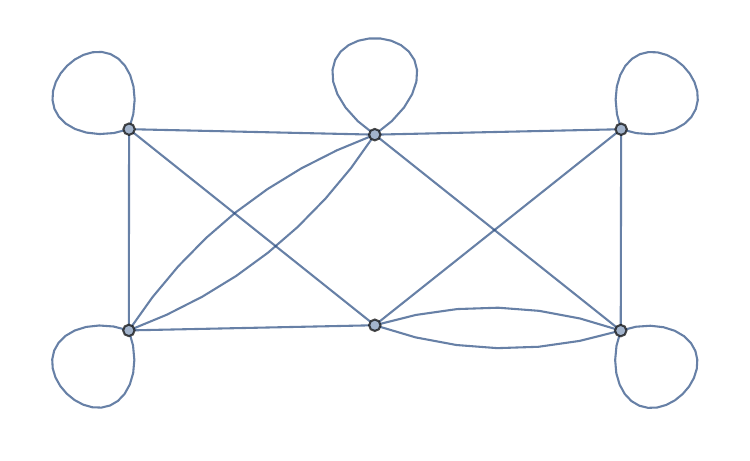}};
(-35,-20)*+{\scriptstyle 11xx};
(0,-20)*+{\scriptstyle x1x1};
(47,-13)*+{\scriptstyle xx11};
(-35,20)*+{\scriptstyle x11 x};
(10,20)*+{\scriptstyle  1x1x};
(47,13)*+{\scriptstyle 1xx1};
(-55,-29)*+{\scriptstyle \frac{1}{2}};
(56,-25)*+{\scriptstyle \frac{1}{2}};
(-57,24)*+{\scriptstyle  1};
(-9,25)*+{\scriptstyle   \frac{1}{2}};
(40,25)*+{\scriptstyle  \frac{1}{2}};
(-40,-31)*+{\scriptstyle \textcolor[rgb]{0.00,0.07,1.00}{m_{12}}};
(-20,-20)*+{\scriptstyle \textcolor[rgb]{0.00,0.07,1.00}{m_{16}}};
(-6,-7)*+{\scriptstyle \textcolor[rgb]{0.00,0.07,1.00}{m_{36}}};
(-45,-2)*+{\scriptstyle \textcolor[rgb]{0.00,0.07,1.00}{m_{13}}};
(-30,-8)*+{\scriptstyle \textcolor[rgb]{0.00,0.07,1.00}{2m_{23}}};
(20,-17)*+{\scriptstyle \textcolor[rgb]{0.00,0.07,1.00}{2m_{23}}};
(38,-29)*+{\scriptstyle\textcolor[rgb]{0.00,0.07,1.00}{m_{36}}};
(-38,29)*+{\scriptstyle \textcolor[rgb]{0.00,0.07,1.00}{2m_{23}}};
(12,25)*+{\scriptstyle \textcolor[rgb]{0.00,0.07,1.00}{m_{13}}};
(58,17)*+{\scriptstyle \textcolor[rgb]{0.00,0.07,1.00}{m_{16}}};
(-20,20)*+{\scriptstyle \textcolor[rgb]{0.00,0.07,1.00}{m_{12}}};
(45,-2)*+{\scriptstyle \textcolor[rgb]{0.00,0.07,1.00}{m_{13}}};
(35,-7)*+{\scriptstyle \textcolor[rgb]{0.00,0.07,1.00}{m_{16}}};
(6,-7)*+{\scriptstyle \textcolor[rgb]{0.00,0.07,1.00}{m_{12}}};
(25,20)*+{\scriptstyle \textcolor[rgb]{0.00,0.07,1.00}{m_{36}}};
\endxy
\]
where we have labelled each edge by the $m_{ij}$ that produced it.  
One can check that the adjacency and degree matrix take the form
\[
D_G =  
\begin{pmatrix}
9/2 & 0 & 0 & 0 & 0 & 0 \\
0 & 11/2 & 0 & 0 & 0 & 0\\
0 & 0 & 4 & 0 & 0 & 0 \\
0 & 0 & 0 & 7/2 & 0 & 0 \\
0 & 0 & 0 & 0 & 5 & 0 \\
0 & 0 & 0 & 0 & 0 & 9/2 \\
\end{pmatrix}
, \qquad 
A_G =  
\begin{pmatrix}
1/2 & 2 & 1 & 0 & 1 & 0 \\
2 & 1/2 & 1 & 1 & 0 & 1 \\
1 & 1 & 1 & 0 & 1 & 0 \\
0 & 1 & 0 &1/2 & 1 & 1 \\
1 & 0 & 1 & 1 & 0 & 2 \\
0 & 1 & 0 & 1 & 2 & 1/2 \\
\end{pmatrix}
\]
so that $\widetilde\Delta_{0,0}(D_{6_3})=Q(G_{2,2}(D_{6_3}))=D_G+A_G$. 
\end{example}

\begin{thm}
  Let $Q(G_{n,k}(D))$ be the signless Laplacian associated to $G_{n,k}(D)$. Then $Q(G_{n,k}(D))=\widetilde\Delta_{0,k-n}(D)$.    
\end{thm}

\begin{proof}
The Laplacian $\widetilde\Delta_{0,k-n}(D)$ is a sum of the contributions of each of the $N$ terms from Equation~\eqref{eq:0Lsplit}.  Decompose the Laplacian $\widetilde\Delta_0(D)$ into a sum of $N=p+t$ maps 
\begin{equation}  \label{eq:0Lsplit-bis}
  \widetilde\Delta_0(D) = \sum_{a=1}^{p} \mathfrak{m}_{i_a, j_a} + \sum_{b=1}^t \mathfrak{s}_{i_b} 
\end{equation}
as in  Equation~\eqref{eq:0Lsplit}.   By definition,  $G_{n,k}(D)$ has 
an edge from $v=v_1 \dots v_i \dots v_j \dots v_\ell$ to $v'=v_1 \dots v_j \dots v_i \dots v_\ell$ for each $\mathfrak{m}_{ij}$ in Equation~\eqref{eq:0Lsplit-bis} with $v_iv_j \in \{1x, x1\}$.  Hence, the entry of the adjacency matrix $A$ takes the form 
\[
w_{vv'} =  \sum_{a=1}^{p}
        \left(\delta_{v_{i_a},1} \delta_{v_{j_a},x} \delta_{v'_{i_a},x} \delta_{v'_{j_a},1} \right)  
        +
        \left(\delta_{v_{i_a},x} \delta_{v_{j_a},1} \delta_{v'_{i_a},1} \delta_{v'_{j_a},x} \right)  
\]
for $v \neq v'$, 
and for $v'=v$ we have a weighted loop contributions from $\mathfrak{m}_{ij}$ when $v_iv_j=11$ and always for $\mathfrak{s}_i$ depending on the entry for $v_i$, 
\[
w_{vv} = \sum_{a=1}^{p}\frac{1}{2}\delta_{v_{i_a},1} \delta_{v_{j_a},1}   + \sum_{b=1}^t (\delta_{v_{i_b},1} +\frac{1}{2}\delta_{v_{i_b},x}  ) .
\]
Hence, $Q=A+D$ for $G_{n,k}(D)$ has entries 
\begin{align}
Q_{vv'} & =   w_{vv'} + \delta_{vv'}\sum_{v''} w_{vv''} = 
\delta_{vv'} \left( 2w_{vv} + \sum_{v''\neq v} w_{vv''} \right) + (1-\delta_{vv'})w_{vv'}  
\end{align}
so that  
\begin{align*}
    Q_{vv} &= \sum_{a=1}^{p} \delta_{v_{i_a},1} \delta_{v_{j_a},1}   + \sum_{b=1}^t (2\delta_{v_{i_b},1} 
     + \delta_{v_{i_b},x}  ) 
     \\  &\quad+ \sum_{v''\neq v} \sum_{a=1}^{p}
        \left(\delta_{v_{i_a},1} \delta_{v_{j_a},x} \delta_{v''_{i_a},x} \delta_{v''_{j_a},1} \right)  
        +
        \left(\delta_{v_{i_a},x} \delta_{v_{j_a},1} \delta_{v''_{i_a},1} \delta_{v''_{j_a},x} \right) 
        \\
     &=\sum_{a=1}^{p} \left( \delta_{v_{i_a},1} \delta_{v_{j_a},1}   
      + \sum_{v''\neq v}  
        \left(\delta_{v_{i_a},1} \delta_{v_{j_a},x} \delta_{v''_{i_a},x} \delta_{v''_{j_a},1} \right)  
        +
        \left(\delta_{v_{i_a},x} \delta_{v_{j_a},1} \delta_{v''_{i_a},1} \delta_{v''_{j_a},x} \right)
     \right) \\
     &\quad 
     + \sum_{b=1}^t (2\delta_{v_{i_b},1} + \delta_{v_{i_b},x}  )
\end{align*}
Examining the terms in the summation over $p$,  each $\mathfrak{m}_{ij}$ in Equation~\eqref{eq:0Lsplit-bis} contributes one to the diagonal of $Q$ if $v_iv_j=11$, and one to the diagonal for each $v''$ with $v_iv_j=1x$ and $v''_iv_j''=x1$ or  $v_iv_j=x1$ and $v''_iv_j''=1 x$.  The off-diagonal entries $Q_{vv'}=w_{vv'}$ for $v\neq v'$ similarly have a contribution of one if $v_iv_j=1x$ and $v'_iv_j'=x1$ or  $v_iv_j=x1$ and $v'_iv_j'=1 x$.   Each term in the summation over $t$ contributes 2 to the diagonal for each $\mathfrak{s}_{i}$ in Equation~\eqref{eq:0Lsplit-bis} if $v_i=1$ and 1 if $v_i = x$.  

Comparing the contributions for each of the terms in the $p$ and $t$ summation with the action of $\mathfrak{m}_{ij}$ and $\mathfrak{s}_i$ in Equation~\eqref{eq:Delta0-bar} completes the proof.  

\end{proof}

% - - - - - - - - - - - - - - - - - -
\subsection{An upper bound for \texorpdfstring{$\Psi$}{Psi}}\label{ssec:psi-bound}
% - - - - - - - - - - - - - - - - - -

While directly computing $\Psi$ for a given graph $G_{n,k}(D)$ can be challenging in general, we can more easily compute the quantity $\Psi_V$ corresponding to $S=V$.  This is not the minimal value over all subsets in general, but it does provide an upper bound for $\Psi$, and hence, the smallest eigenvalue of $\widetilde\Delta_{0,k-n}(D)$ by Theorem~\ref{thm:Desai}. 

To compute the quantity $\Psi_V$, observe we have $|V| = \binom{n+k}{k}$ and $\operatorname{cut}(V)=0$.  The quantity $e_{\min}(V)$ has contributions from the sum of the weights over all weighted loops in $G$, since these all must be removed to form a bipartite graph, and contributions from edges that must be removed to make $G$ bipartite.  Example~\ref{example:6-3-graph} shows that the graphs $G_{n,k}(D)$ are, in general, non-bipartite even after forming the graph $G'_{n,k}(D)$ obtained from $G_{n,k}(D)$ by forgetting loops at vertices.

% - - - - - - - - - - - - - - - - - -
\subsubsection{Counting loops}
% - - - - - - - - - - - - - - - - - -

The total weight of loops in $G_{n,k}(D)$ admits a closed-form count in terms of $\ell$, $n$, $k$, and the multiplication/comultiplication counts $p, t$ of~\eqref{eq:0Lsplit}. Recall from the construction that an $\mathfrak{m}_{ij}$ contributes a loop of weight $1/2$ at $v$ exactly when $v_i = v_j = 1$, and a $\mathfrak{s}_i$ contributes a loop at $v$ of weight $1$ if $v_i = 1$ and of weight $1/2$ if $v_i = x$.

We count contributions from multiplication terms first. For a fixed $\mathfrak{m}_{ij}$, the number of vertices $v$ with $v_i = v_j = 1$ depends only on $\ell$, $n$, $k$, not on the specific indices: with the $i$th and $j$th entries forced to $1$, the remaining $\ell - 2$ entries must contain exactly $k - 2$ ones and $n$ x's, so there are $\binom{\ell - 2}{k - 2}$ such vertices, each carrying a loop of weight $1/2$. Summing over the $p$ multiplication terms, the total weight of loops from multiplications is 
\[
 \frac{p}{2} \binom{\ell-2}{k-2}.
\]

Next, we count weighted sums of loops corresponding to the $t$ split terms in the decomposition of $\widetilde\Delta_0(D)$.  To compute this, we compute how many of the $\binom{n+k}{k}$ sequences have $v_i=1$ and how many have $v_i = x$.  Again, these numbers are independent of $i$.  There are $\binom{\ell-1}{k-1}$ such sequences with $v_i=1$ and $\binom{\ell-1}{k}$ with $v_i = x$.  Hence, the weighted sum of loops from the $t$ split terms is 
\[
t \binom{\ell-1}{k-1} + \frac{t}{2} \binom{\ell-1}{k}.
\]

Hence, we have proven the following Lemma.
\begin{lemma} \label{lem:loop-count}
The sum of weights of loops in $G_{n,k}(D)$ is given by 
\[
 \frac{p}{2} \binom{\ell-2}{k-2} + t \binom{\ell-1}{k-1} + \frac{t}{2} \binom{\ell-1}{k}
\]
where $p$ is the number of multiplications and $t$ is the number of comultiplications in the decomposition of $\widetilde\Delta_0(D)$ from Equation~\eqref{eq:0Lsplit}.
\end{lemma}

\begin{example}
In Example~\ref{example:6-3-graph} of the graph $G_{2,2}(D_{6_3})$ for the standard diagram $D_{6_3}$ at $k=2$, we have $\ell=4$, $k=n=2$, $p=6$, and $t=0$.  The weighted sum of the loops appearing in $G_{2,2}(D_{6_3})$ is 
\[
\frac{6}{2} \binom{2}{2} = 3.
\]
\end{example}

% - - - - - - - - - - - - - - - - - -
\subsubsection{Odd cycles in $G_{n,k}(D)$  } \label{subsec:odd-cycle}
% - - - - - - - - - - - - - - - - - -
The number of odd cycles in $G_{n,k}(D)$ can also be determined from the decomposition of $\widetilde\Delta_0(D)$.  Let $G'_{n,k}(D)$ denote the graph $G_{n,k}(D)$ with all loops removed.  Since comultiplication only contributes loops, we need only consider the $p$ multiplications appearing in the decomposition of $\widetilde\Delta_0(D)$
\[
 \sum_{a=1}^{p} \mathfrak{m}_{i_a, j_a}
\]
The cycle structure of $G'_{n,k}(D)$ is the same as the cycle structure in the set of unordered pairs $\left\{(i_a,j_a) \right\}_{a=1}^p$. 

Let $\Upsilon$ denote the finite set of simple odd cycles in the index graph whose vertices are the tensor-factor labels and whose edge multiset is $\{(i_a,j_a)\}_{a=1}^p$. Thus an element $u\in\Upsilon$ has the form $(a_0,a_1)(a_1,a_2)\cdots(a_{2r},a_0)$, with distinct vertices $a_0,\ldots,a_{2r}$ and no repeated edge. In the computation of $e_{\min}(V)$, we must remove edges to make $G'_{n,k}(D)$ bipartite. The graph $G'_{n,k}(D)$ contains odd cycles arising from each $u\in\Upsilon$, corresponding to the two choices $1x$ and $x1$. Removing a hitting set of ordinary edges for these cycles, together with all loops, makes the graph bipartite.

Some odd cycles in $\Upsilon$ may intersect, and some edges may have weightings, so selecting a minimal size set of weighted edges to remove from $G'_{n,k}(D)$ to make it bipartite can be nontrivial.  Clearly, $e_{\min}(V)$ is less than or equal to the sum over $u \in \Upsilon$ of the minimal weighting of an edge appearing in $u$.  Specifically, for  $u= (a_0,a_1)(a_1,a_2)\cdots(a_{2r},a_0) \in \Upsilon$, define  $\varrho(u)$ and $\Upsilon_D$ as
\[
\varrho(u)  = \min_{(i,j) \in u} w_{ij} , \qquad  \Upsilon_D:= \sum_{u \in \Upsilon} \varrho(u).
\]  

\begin{lemma} \label{lem:eV-bound}
For the graph $G_{n,k}(D)$,
\begin{align} \label{eq:eminV}
   e_{\min}(V) &=  \frac{p}{2} \binom{\ell-2}{k-2} + t \binom{\ell-1}{k-1} + \frac{t}{2} \binom{\ell-1}{k} + e'_{\min}(V) \\
   &\leq  \frac{p}{2} \binom{\ell-2}{k-2} + t \binom{\ell-1}{k-1} + \frac{t}{2} \binom{\ell-1}{k} +2 \Upsilon_D . \nonumber
\end{align}
Here $e'_{\min}(V)$ is the minimum total weight of non-loop edges that must be removed from $G'_{n,k}(D)$ to make it bipartite.
\end{lemma}

\begin{proof}
The loop contribution is exactly the total loop weight computed in Lemma~\ref{lem:loop-count}; each such loop must be removed to make the graph bipartite. The remaining contribution is the minimum total weight $e'_{\min}(V)$ of ordinary edges that must be removed from the loop-free graph $G'_{n,k}(D)$. The bound by $2\Upsilon_D$ follows from the cycle-removal estimate above.
\end{proof}

\begin{proposition} \label{prop:upper}
The smallest eigenvalue $\lambda_{\min}$ of $\widetilde\Delta_{0,k-n}(D)$ satisfies
\[
\lambda_{\min} \leq
\frac{4\left(\frac{p}{2} \binom{\ell-2}{k-2} + t \binom{\ell-1}{k-1} + \frac{t}{2} \binom{\ell-1}{k} +2 \Upsilon_D\right)}{\binom{n+k}{k}} .
\]
\end{proposition}

\begin{proof}
Taking $S=V$ gives $\Psi\leq \Psi_V=e_{\min}(V)/|V|$. The result follows from the upper bound $\lambda_{\min}(Q)\leq 4\Psi$ in Theorem~\ref{thm:Desai} and Lemma~\ref{lem:eV-bound}.
\end{proof}

\begin{example}
In Example~\ref{example:6-3-graph}, the cube degree zero Laplacian of the standard diagram $D_{6_3}$ decomposes as
\[
\widetilde\Delta_0(D_{6_3}) = \mathfrak{m}_{12}+2\mathfrak{m}_{23} + \mathfrak{m}_{13} + \mathfrak{m}_{16} + \mathfrak{m}_{36}.
\]
Moreover,
\[
\Upsilon = \left\{
(1,2)(2,3)(1,3), \; (1,3) (3,6) (1,6)
\right\}.
\]
The edge $(2,3)$ has weight $2$, and $(1,3)$ appears in both cycles. Removing both non-loop edges labeled $(1,3)$ from $G'_{2,2}(D_{6_3})$ produces a bipartite graph, so that $e_{\min}(V)=2+\#\text{loops}=5$. The coarser estimate using $\Upsilon_D=2$ gives $e_{\min}(V)\leq 4+\#\text{loops}=7$, and hence Proposition~\ref{prop:upper} gives $\lambda_{\min}\leq 4\cdot 7/6=14/3$. Using the exact value $\Psi_V=e_{\min}(V)/|V|=5/6$ instead gives the sharper upper bound $\lambda_{\min}\leq 10/3$. The actual value is $\lambda_{\min}=0.78916$.
\end{example}

\begin{remark}
In this section, we have focused on constructing graphs for Khovanov homology at cube degree zero.  A similar technique can be used in maximal cube degree using the maps
 \begin{alignat*}{3}
  \Delta_V\Delta_V^{\dagger} \maps V \otimes V&\to V \otimes V   \qquad \qquad    
        &&  m m^{\dagger} \maps V \to V  \\
 \ket{11} &\mapsto 0 \qquad \qquad    
    && \quad\; \ket{1} \mapsto \ket{1 }   \\
 \ket{1x} &\mapsto \ket{1x} + \ket{x1} \qquad \qquad    && \quad\; \ket{x} \mapsto 2\ket{x}\\
\ket{x1}  &\mapsto \ket{1x} + \ket{x1}  \\
\ket{xx}  &\mapsto \ket{xx} \\
\end{alignat*}
This essentially just swaps the role of $1$ and $x$ in our analysis. 
\end{remark}

% - - - - - - - - - - - - - - - - - -
\subsection{Graphs for twisted unknots}\label{ssec:twisted-unknot-graphs}
% - - - - - - - - - - - - - - - - - -
We first examine the twisted-unknot diagrams $TU_N$ in the unnormalized cube grading of Section~\ref{ssec:khovanov-complex}. The later torus-knot and twist-knot results are stated in normalized Khovanov bidegrees. 
\[
TU_N \;\; := \;\; 
\hackcenter{   \begin{tikzpicture}[scale=.6]
      \draw [ very thick]    (1,1) .. controls +(0,.35) and +(0,-.35) .. (2,2);
  \draw [ very thick]    (1,2) .. controls +(0,.35) and +(0,-.35) .. (2,3);
   \draw [ very thick]    (1,4) .. controls +(0,.35) and +(0,-.35) .. (2,5);
  \path [fill=white] (1.35,1) rectangle (1.65,5);
  \draw [ very thick]    (2,1) .. controls +(0,.35) and +(0,-.35) .. (1,2);
  \draw [ very thick]    (2,2) .. controls +(0,.35) and +(0,-.35) .. (1,3);
    \draw [ very thick]    (2,4) .. controls +(0,.35) and +(0,-.35) .. (1,5);
     \draw [ very thick]    (2,5) .. controls +(0,.5) and +(0,.5) .. (1,5);
     \draw [ very thick]    (2,1) .. controls +(0,-.5) and +(0,-.5) .. (1,1);
      \node at (1.5,3.75) {$\vdots$};
    \end{tikzpicture}  } 
    \qquad \qquad 
    \sigma^N \;\; := \;\; 
\hackcenter{   \begin{tikzpicture}[scale=.6]
      \draw [ very thick]    (1,1) .. controls +(0,.35) and +(0,-.35) .. (2,2);
  \draw [ very thick]    (1,2) .. controls +(0,.35) and +(0,-.35) .. (2,3);
   \draw [ very thick]    (1,4) .. controls +(0,.35) and +(0,-.35) .. (2,5);
  \path [fill=white] (1.35,1) rectangle (1.65,5);
  \draw [ very thick]    (2,1) .. controls +(0,.35) and +(0,-.35) .. (1,2);
  \draw [ very thick]    (2,2) .. controls +(0,.35) and +(0,-.35) .. (1,3);
    \draw [ very thick]    (2,4) .. controls +(0,.35) and +(0,-.35) .. (1,5); 
      \node at (1.5,3.75) {$\vdots$};
    \end{tikzpicture}  }
\]
Twisted unknots are much easier to analyze than general knots because the decomposition of $\widetilde\Delta_0(TU_N)$ from Equation \eqref{eq:0Lsplit} takes an especially nice form.  The all 0-resolution of $TU_N$ has $N+1$ circles and the Laplacian decomposes as
\[
\widetilde\Delta_{0}(TU_N) = \sum_{i=1}^N \mathfrak{m}_{i,i+1}.    
\]

\begin{proposition} \label{prop-TU-bipartite}
Removing all weighted loops from $G_{n,k}(TU_N)$ to form $G'_{n,k}(TU_N)$ gives a bipartite connected graph. 
\end{proposition}

\begin{proof}
Following the analysis in Section~\ref{subsec:odd-cycle}, it is clear that for twisted unknots, $\Upsilon = \emptyset$ since there are no cycles in the set $\{(i,i+1) \}_{i=1}^N$.  Another way to see this is to write any vertex $v$ in $G_{n,k}(TU_N)$ as $w(1 \dots 1 x \dots x)$ where $w$ is the permutation mapping $1 \dots 1 x \dots x$ (with $k$ entries $1$ and $n$ entries x) to $v$ without swapping the set of $1$s or the set of $x$'s; equivalently, $w$ ranges over coset representatives of $\mathfrak{S}_k \times \mathfrak{S}_n \subset \mathfrak{S}_{N+1}$.  Let $\ell(w)$ denote the length of any reduced expression of $w$ as a product of elementary transpositions.  

To see that $G'_{n,k}(TU_N)$ bipartite, observe that edges  connect vertices $w(1 \dots 1 x \dots x)$ to vertices $w'(1 \dots 1 x \dots x)$ only if $\ell(w)=\ell(w')\pm 1$. Moreover, the graph structure is determined by the Bruhat order of the symmetric group.  Specifically, for $w',w$ two coset representatives of $\mathfrak{S}_n \times \mathfrak{S}_k \in \mathfrak{S}_{N+1}$, there will be an edge from the vertex labeled by $w$ to the vertex labeled $w'$ if and only if $w'=s_iw$ is a reduced expression for $w'$, where $s_i$ is a simple transposition.   This also shows that the resulting graph will be connected. 
\end{proof}

\begin{remark}
The proof of Proposition~\ref{prop-TU-bipartite} shows that the graphs $G_{n,k}(TU_N)$  can be interpreted as Schreier coset graphs $Sch(\mathfrak{S}_{N+1},\mathfrak{S}_k \times \mathfrak{S}_n, x)$ for the symmetric group $\mathfrak{S}_{N+1}$ with subgroup $\mathfrak{S}_k \times \mathfrak{S}_n$ and generating set $S=\{s_1, \dots , s_N\}$ the set of simple transpositions.  When $x$ generates $G/H$ the graph $Sch(\mathfrak{S}_{N+1},\mathfrak{S}_k \times \mathfrak{S}_n, x)$ is connected~\cite[Section 6.3]{MR2882891}.  
\end{remark}

Proposition~\ref{prop-TU-bipartite} simplifies the computation of $\Psi$ for twisted unknots as we only need to count weights for loops at vertices in subsets $S\subset V$. 
Given a subset $S\subset V$, let $\operatorname{loop}(S)$ denote the weighted sum of loops at the vertices in $S$.  Then $e_{\min}(S)=\operatorname{loop}(S)$ and $\Psi_S=(\operatorname{loop}(S)+\operatorname{cut}(S))/|S|$. 
  Lemma~\ref{lem:loop-count} implies that for $G_{n,k}(TU_N)$ we have 
 \begin{equation} \label{eq:PsiV}
    \Psi_V =  \frac{N}{2}\frac{\binom{N-1}{k-2}}  {\binom{N+1}{k}} = \frac{k(k-1)}{2(N+1)}
 \end{equation}

The following is immediate from Theorem~\ref{thm:Desai}. 

\begin{corollary} \label{cor:gapbound}
Let $n+k=N+1$.  The minimal eigenvalue $\lambda_{\min}$ of $\widetilde\Delta_{0,k-n}(TU_N)$  satisfies
\[
\lambda_{\min} \leq \frac{2k(k-1)}{(N+1)}.
\]
In particular, when $k=0$, and $k=1$, the kernel of the Laplacian is nontrivial.
\end{corollary}

\begin{remark}
Corollary~\ref{cor:gapbound} implies that the bound from Proposition \ref{prop:upper} cannot be used to find an example of a knot with an exponentially small spectral gap using $TU_N$ in cube degree zero. 
\end{remark}

Computing $\Psi$ explicitly, and hence a lower bound, even for $TU_N$, can be challenging.  The subset $S$ giving the minimal value of $\Psi_S$ is, in general, not the entire set of vertices $V$. 
 Given subsets $S\subset S' \subset V$, we write $\operatorname{cut}(S,S')$ for $\operatorname{cut}(S)$ thought of as a subset of the restricted graph $G_{S'}$.  Clearly, 
 \begin{equation}  \label{eq:loopcut-ind}
\operatorname{loop}(S')=\operatorname{loop}(S)+\operatorname{loop}(S'\setminus S), \qquad \operatorname{cut}(S')=\operatorname{cut}(S)-\operatorname{cut}(S,S')+\operatorname{cut}(S'\setminus S,V\setminus S),
\end{equation}
where $S'\setminus S$ denotes set difference.

\begin{thm} \label{thm:subsets}
For the diagram $TU_N$ and $S\subset S' \subset V$, we have $\Psi_{S} < \Psi_{S'}$ if and only if $t > |S'\setminus S|\Psi_{S'}$ where $t=\operatorname{loop}(S'\setminus S)-\operatorname{cut}(S,S')+\operatorname{cut}(S'\setminus S,V\setminus S)$.
\end{thm}

\begin{proof}
By Proposition~\ref{prop-TU-bipartite} and Equation~\eqref{eq:loopcut-ind},
\[
\Psi_S
=
\frac{\operatorname{loop}(S')+\operatorname{cut}(S')-t}
{|S'|-|S'\setminus S|},
\qquad
\Psi_{S'}
=
\frac{\operatorname{loop}(S')+\operatorname{cut}(S')}{|S'|}.
\]
Therefore $\Psi_S<\Psi_{S'}$ if and only if
\[
|S'|t
>
|S'\setminus S|\bigl(\operatorname{loop}(S')+\operatorname{cut}(S')\bigr),
\]
which is equivalent to $t>|S'\setminus S|\Psi_{S'}$.
\end{proof}

We now focus on $G_{n,2}(TU_N)$.  For $k=2$, as in the basis used in Theorem~\ref{thm:tu-gap}, it is convenient to indicate the location of the two $1$'s by a 2-tuple $\{ a_1,a_2\}$ with $1 \leq a_1 < a_2 \leq N+1$.
Any such $v=\{a_1,a_2\}$ is connected to between 4 and 1 edges $\{a_1\pm 1, a_2\pm 1 \}$ in $G_{n,2}(TU_N)$ depending on if $a_1\neq 1$, $a_2\neq N+1$ or $a_2=a_1+1$.   Recall our convention that $\{a,b\}=0$ if $a=b$, $a<1$ or $b>N+1$. 

Observe that a vertex $v$ of $G_{n,2}(TU_N)$ has a loop if and only if the $k=2$ $1$'s appearing in $v$ are adjacent, so that $v=\{a,a+1\}$ for $1\leq a \leq N$.
Vertices connected to a vertex  $v=\{ a, a+1\}$ with a loop will never have a loop since $\{a-1,a+1\}$ and $\{a, a+2\}$ never have adjacent $1$'s. 

\begin{lemma} \label{lem:loopLcut}
For any proper subset $T \subset V$ in $G_{n,2}(TU_N)$ we have $\operatorname{loop}(T)<\operatorname{cut}(T)$.
\end{lemma}

\begin{proof}
Observe that if $v=\{a,a+1\}$ is a vertex with a weight $1/2$ loop it will have at least one edge $(v,v')$ in $G_{n,2}(TU_N)$. The only way this vertex can contribute positively to the quantity $\operatorname{loop}(T)-\operatorname{cut}(T)$ is if all of the edges $(v,v')$ from $v$ connect to vertices $v'$ that are also in $T$.  Loop edges with $k=2$ always connect to non-loop edges $\{a-1,a+1\}$ or $\{ a,a+2\}$.  Non-loop vertices contribute nothing to $\operatorname{loop}(T)$ but always have at least two incident edges.  Hence, the vertex $v$ contributes positively only if every vertex in $V$ is in $T$, contradicting that $T$ was proper. 
\end{proof}

\begin{corollary} \label{cor:bound}
For $k=0,1,2$, one has $\Psi = \Psi_V$ for $G_{n,k}(TU_N)$. Hence, 
\[
 \frac{\Psi_V^2}{4 d^*} =  \frac{1}{16}\left(\frac{k(k-1)}{2(N+1)}\right)^2 \leq \lambda_{\min} \leq  \frac{2k(k-1)}{(N+1)}=4\Psi_V.
\]  
where $d^*$ is the maximal degree of a vertex in $G_{n,k}(TU_N)$.   
\end{corollary}

\begin{proof}
Theorem~\ref{thm:subsets} can be helpful in computing the subset $S$ minimizing $\Psi_S$.  Applying Theorem~\ref{thm:subsets} when $S'=V$ and using Equation \eqref{eq:PsiV} implies that $\Psi_S < \Psi_V$ if and only if
\begin{equation} \label{eq:t}
    t = \operatorname{loop}(V\setminus S)-\operatorname{cut}(S)=\operatorname{loop}(V\setminus S)-\operatorname{cut}(V\setminus S)  > |V-S| \frac{k(k-1)}{2(N+1)} 
\end{equation}
This implies that if one is trying to find a subset $S$ with a smaller value of $\Psi_S$ than $\Psi_V$, one must select a subset of the vertices of $V$ with the property that the loop count in the vertices we remove is larger (by $|V-S| \frac{k(k-1)}{2(N+1)}$) than the connections from $S$ to the removed vertices in the graph.  For $k=0$ or $k=1$ this is impossible. For $k=2$, taking $T=V/S$, it follows that $t\leq 0$, so that $t$ never satisfies the inequality Equation \eqref{eq:t}. 
\end{proof}

\begin{example}[$N=3$ Twisted unknot]
Let $n=k=2$ with $n+k=N+1=4$.  Then the Laplacian has the form 
\[
\widetilde\Delta_{0,0}(TU_3) =   \mathfrak{m}_{12} + \mathfrak{m}_{23}+\mathfrak{m}_{34} = 
\left(
\begin{array}{cccccc}
 2 & 1 & 0 & 0 & 0 & 0 \\
 1 & 3 & 1 & 1 & 0 & 0 \\
 0 & 1 & 3 & 0 & 1 & 0 \\
 0 & 1 & 0 & 2 & 1 & 0 \\
 0 & 0 & 1 & 1 & 3 & 1 \\
 0 & 0 & 0 & 0 & 1 & 2 \\
\end{array}
\right)
\]
\[G_{2,2}(TU_3) = 
\xy
(0,0)*+{\includegraphics[width=3.5in]{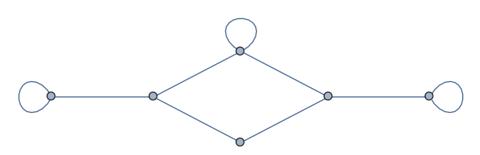} };
(-35,-8)*+{\scriptstyle 11xx};
(-18,-8)*+{\scriptstyle  1x1x};
(18,-8)*+{\scriptstyle x1x1};
(35,-8)*+{\scriptstyle xx11};
(-8,6)*+{\scriptstyle x11 x};
(0,-15)*+{\scriptstyle 1xx1};
(-37,2)*+{\scriptstyle \frac{1}{2}};
(38,2)*+{\scriptstyle \frac{1}{2}}; 
(6,10)*+{\scriptstyle   \frac{1}{2}};
(-45,-5)*+{\scriptstyle \textcolor[rgb]{0.00,0.07,1.00}{m_{12}}};
(-25,-1)*+{\scriptstyle \textcolor[rgb]{0.00,0.07,1.00}{m_{23}}};
(26,-1)*+{\scriptstyle \textcolor[rgb]{0.00,0.07,1.00}{m_{23}}};
(-7,-11)*+{\scriptstyle \textcolor[rgb]{0.00,0.07,1.00}{m_{34}}};
(-11,2)*+{\scriptstyle \textcolor[rgb]{0.00,0.07,1.00}{m_{12}}};
(9,-11)*+{\scriptstyle \textcolor[rgb]{0.00,0.07,1.00}{m_{12}}};
(10,2)*+{\scriptstyle \textcolor[rgb]{0.00,0.07,1.00}{m_{34}}};
(45,-5)*+{\scriptstyle \textcolor[rgb]{0.00,0.07,1.00}{m_{34}}};
(0,12)*+{\scriptstyle \textcolor[rgb]{0.00,0.07,1.00}{m_{23}}};
\endxy
\]
the minimal eigenvalue is $\lambda_{\min}(Q)=\frac{1}{2} \left(3-\sqrt{5}\right)=0.381966$.  In this case, the minimum for $\Psi$ occurs for $S=V$ where $e_{\min}(S)=\operatorname{loop}(S)=3/2$, $\operatorname{cut}(S)=0$, and $|S|=6$, so that $\Psi=1/4$. Then Theorem~\ref{thm:Desai} implies that 
\[
    \frac{\Psi^2}{4d^*} =   \frac{(1/4)^2}{12} = \frac{1}{192} =0.005208\leq 0.381966= \lambda_{\min}(Q) \leq 4\Psi = 1. 
\]
\end{example}

\subsubsection{The exact spectral gap of \texorpdfstring{$TU_N$}{TU\textunderscore N}}\label{sssec:tu-min-gap}

The bound of Corollary~\ref{cor:gapbound} can be sharpened to a closed-form expression in the $q$-degree where the minimum gap is conjectured to occur.

For a real symmetric matrix $L$ indexed by a finite set $V$, define the
\emph{support graph} of $L$ to be the graph on $V$ with an edge $\{u,v\}$
for distinct vertices $u,v$ whenever $L_{uv}\neq 0$.

\begin{lemma}\label{lem:positive-eigenvector}
Let $L$ be a real symmetric matrix indexed by a finite set $V$ with
$L_{uv}\leq 0$ for all $u\neq v$, and suppose that the support graph of $L$
is connected. If
\[
LF=\lambda F
\]
for a vector $F$ with $F(u)>0$ for every $u\in V$, then $\lambda$ is the
smallest eigenvalue of $L$, and this eigenvalue is simple.
\end{lemma}

\begin{proof}
Let $\mu$ be the smallest eigenvalue of $L$, and let $G$ be a unit
$\mu$ eigenvector. Write $|G|$ for the vector with entries $|G(u)|$. Since
$L$ is symmetric, $L_{uv}\leq 0$ for $u\neq v$, and
$|G(u)||G(v)|\geq G(u)G(v)$, we have
\[
\begin{aligned}
\langle |G|,L|G|\rangle
&=
\sum_u L_{uu}G(u)^2
+
2\sum_{u<v}L_{uv}|G(u)||G(v)| \\
&\leq
\sum_u L_{uu}G(u)^2
+
2\sum_{u<v}L_{uv}G(u)G(v) \\
&=
\langle G,LG\rangle
=
\mu.
\end{aligned}
\]
The Rayleigh principle gives the reverse inequality, so equality holds and
$|G|$ is also a $\mu$ eigenvector.

Suppose that $|G(u)|=0$ for some $u\in V$. The $u$th entry of
$(L-\mu)|G|=0$ gives
\[
\sum_{v\neq u}L_{uv}|G(v)|=0.
\]
Every summand is nonpositive. Moreover, $L_{uv}<0$ whenever $u$ and $v$ are
adjacent in the support graph. It follows that $|G(v)|=0$ for every neighbor
$v$ of $u$. Connectedness then forces $G=0$, a contradiction. Thus every
$\mu$ eigenvector has no zero entries.

Equality in the first displayed inequality also implies
\[
|G(u)||G(v)|=G(u)G(v)
\]
on every edge of the support graph. Hence $G$ has the same sign at adjacent
vertices. Since the support graph is connected, $G$ has one strict sign on
all of $V$.
After replacing $G$ by $-G$ if necessary, we may assume that $G(u)>0$ for
every $u$. If $\lambda\neq\mu$, symmetry of $L$ would imply
$\langle F,G\rangle=0$, which is impossible since both vectors are strictly
positive. Therefore $\lambda=\mu$.

Finally, if the $\mu$ eigenspace had dimension greater than one, it would
contain a nonzero vector orthogonal to $F$. This is impossible because every
$\mu$ eigenvector has one strict sign. Thus $\mu$ is simple.
\end{proof}

\begin{thm}\label{thm:tu-gap}
For every $N\geq 1$, the spectral gap of the Khovanov Laplacian of the
twisted unknot $TU_N$ in bidegree $(0,3-N)$ is
\[
\operatorname{gap}\bigl(\widetilde\Delta_{0,3-N}(TU_N)\bigr)
=
2-2\cos\left(\frac{\pi}{N+2}\right)
=
4\sin^2\left(\frac{\pi}{2(N+2)}\right).
\]
The smallest eigenvalue is simple.
\end{thm}

\begin{proof}
The all zero resolution of $TU_N$ consists of $N+1$ circles. Number them
$1,\dots,N+1$ along the diagram so that crossing $i$ joins circle $i$ to
circle $i+1$. In unnormalized $q$ degree $3-N$, an enhanced state carries
the label $1$ on exactly two circles and the label $x$ on the remaining
$N-1$ circles. For $1\leq a<b\leq N+1$, let $e_{a,b}$ denote the state with
$1$ on circles $a$ and $b$. The vectors
\[
\{e_{a,b}\mid 1\leq a<b\leq N+1\}
\]
form an orthonormal basis of $\widetilde C^{0,3-N}(TU_N)$. We set
$e_{a',b'}=0$ whenever $a'<1$, $b'>N+1$, or $a'=b'$.

By \eqref{eq:Delta0-bar},
\[
\widetilde\Delta_0(TU_N)
=
\sum_{i=1}^N\mathfrak{m}_{i,i+1},
\]
where $\mathfrak{m}_{i,i+1}=m^\dagger m$ acts on the tensor factors
corresponding to circles $i$ and $i+1$ by
\[
|11\rangle\longmapsto |11\rangle,\qquad
|1x\rangle\longmapsto |1x\rangle+|x1\rangle,
\]
\[
|x1\rangle\longmapsto |1x\rangle+|x1\rangle,\qquad
|xx\rangle\longmapsto 0,
\]
and acts by the identity on the remaining factors.

On the state $e_{a,b}$, the summand associated to crossing $i$ contributes
nothing if neither adjacent circle carries a $1$. It contributes $e_{a,b}$
if both circles carry a $1$. If exactly one carries a $1$, it contributes
$e_{a,b}$ together with the state obtained by moving that label across
crossing $i$. Summing over the crossings gives
\begin{equation}\label{eq:tu-laplacian-explicit}
\widetilde\Delta_{0,3-N}(TU_N)e_{a,b}
=
d(a,b)e_{a,b}
+
e_{a-1,b}+e_{a+1,b}+e_{a,b-1}+e_{a,b+1},
\end{equation}
where $d(a,b)$ is the number of crossings adjacent to circle $a$ or circle
$b$.

Circle $a$ meets crossings $a-1$ and $a$, except that circle $1$ meets only
crossing $1$. Similarly, circle $b$ meets crossings $b-1$ and $b$, except
that circle $N+1$ meets only crossing $N$. When $b=a+1$, crossing $a$ is
counted only once. Therefore
\[
d(a,b)
=
4-\delta_{a,1}-\delta_{b,N+1}-\delta_{b,a+1}.
\]

The off diagonal terms in \eqref{eq:tu-laplacian-explicit} connect states
that differ by moving one label across one crossing. We call such a move a
\emph{hop}. Every hop changes the parity of $a+b$. Define
\[
\widetilde e_{a,b}:=(-1)^{a+b}e_{a,b}.
\]
In this basis,
\[
\widetilde\Delta_{0,3-N}(TU_N)\widetilde e_{a,b}
=
d(a,b)\widetilde e_{a,b}
-
\widetilde e_{a-1,b}
-
\widetilde e_{a+1,b}
-
\widetilde e_{a,b-1}
-
\widetilde e_{a,b+1}.
\]

Let $L$ denote the matrix of $\widetilde\Delta_{0,3-N}(TU_N)$ in the basis
$\widetilde e_{a,b}$. Its off diagonal entries are $-1$ or $0$, and its
support graph is the hop graph. This graph is connected: starting from a
pair $(a,b)$, one can first lower $a$ to $1$ and then lower $b$ to $2$.
Thus every pair is connected to $(1,2)$.

Identify
\[
\sum_{a<b}F(a,b)\widetilde e_{a,b}
\]
with the function $F$ on pairs $1\leq a<b\leq N+1$. Then
\begin{equation}\label{eq:L-on-functions}
\begin{aligned}
(LF)(a,b)
={}&
d(a,b)F(a,b)
-F(a-1,b)-F(a+1,b) \\
&-F(a,b-1)-F(a,b+1),
\end{aligned}
\end{equation}
where a term $F(a',b')$ is omitted whenever $e_{a',b'}=0$.

Away from the endpoint conditions $a=1$ and $b=N+1$ and the adjacency
condition $b=a+1$, the operator in \eqref{eq:L-on-functions} is the sum of
two one dimensional path operators. This suggests looking for a solution of
the form
\[
F(a,b)=\alpha(a)+\beta(b).
\]
The endpoint corrections will require
\[
\alpha(0)=\alpha(1),
\qquad
\beta(N+2)=\beta(N+1),
\]
while the adjacency correction will require
\[
\alpha(a+1)+\beta(a)=0.
\]

Set
\[
\theta=\frac{\pi}{N+2},
\qquad
\lambda=2-2\cos\theta,
\]
and define
\[
\alpha(t)=\cos\bigl(\theta(t-\tfrac12)\bigr),
\qquad
\beta(t)=\cos\bigl(\theta(t-N-\tfrac32)\bigr),
\qquad
F(a,b)=\alpha(a)+\beta(b).
\]
The cosine recursion gives
\begin{equation}\label{eq:cos-recursion}
\begin{aligned}
2\alpha(t)-\alpha(t-1)-\alpha(t+1)
&=\lambda\alpha(t),\\
2\beta(t)-\beta(t-1)-\beta(t+1)
&=\lambda\beta(t)
\end{aligned}
\end{equation}
for every $t\in\mathbb Z$. Since cosine is even,
\begin{equation}\label{eq:cos-endpoints}
\alpha(0)=\alpha(1)=\cos(\theta/2),
\qquad
\beta(N+2)=\beta(N+1)=\cos(\theta/2).
\end{equation}
Finally, the sum to product identity gives
\begin{equation}\label{eq:cos-adjacent}
\begin{aligned}
\alpha(t+1)+\beta(t)
&=
2\cos\left(\frac{\theta(2t-N-1)}{2}\right)
 \cos\left(\frac{\theta(N+2)}{2}\right)\\
&=0
\end{aligned}
\end{equation}
for every $t\in\mathbb Z$, since $\theta(N+2)/2=\pi/2$.

To verify the eigenvalue equation, introduce the free operator
\[
\begin{aligned}
\Lambda F(a,b)
:={}&
4F(a,b)
-F(a-1,b)-F(a+1,b)\\
&-F(a,b-1)-F(a,b+1).
\end{aligned}
\]
Equation \eqref{eq:cos-recursion} gives
\[
\Lambda F=\lambda F
\]
for every pair of integers. For $1\leq a<b\leq N+1$, comparison with
\eqref{eq:L-on-functions} gives
\[
\begin{aligned}
(LF)(a,b)-\Lambda F(a,b)
={}&
\delta_{a,1}\bigl(F(0,b)-F(1,b)\bigr)\\
&+
\delta_{b,N+1}\bigl(F(a,N+2)-F(a,N+1)\bigr)\\
&+
\delta_{b,a+1}
\bigl(F(a,a)+F(a+1,a+1)-F(a,a+1)\bigr).
\end{aligned}
\]
The first two terms vanish by \eqref{eq:cos-endpoints}. The remaining bracket
is
\[
F(a,a)+F(a+1,a+1)-F(a,a+1)
=
\alpha(a+1)+\beta(a),
\]
which vanishes by \eqref{eq:cos-adjacent}. Hence
\[
LF=\lambda F.
\]

The sum to product identity also gives
\[
F(a,b)
=
2\cos\left(\frac{\theta(a+b-N-2)}{2}\right)
 \cos\left(\frac{\theta(N+1+a-b)}{2}\right).
\]
For $1\leq a<b\leq N+1$,
\[
|a+b-N-2|\leq N-1,
\qquad
1\leq N+1+a-b\leq N.
\]
The two cosine arguments therefore lie in
\[
\left(-\frac{\pi}{2},\frac{\pi}{2}\right)
\qquad\text{and}\qquad
\left(0,\frac{\pi}{2}\right),
\]
respectively. Thus $F(a,b)>0$ for every $1\leq a<b\leq N+1$.

Lemma~\ref{lem:positive-eigenvector} now implies that
\[
\lambda
=
2-2\cos\left(\frac{\pi}{N+2}\right)
\]
is the smallest eigenvalue of $L$, and that it is simple. Since $\lambda>0$,
the Laplacian has no kernel in this bidegree, so $\lambda$ is its spectral
gap. In the original enhanced state basis, the corresponding eigenvector is
\[
\sum_{a<b}(-1)^{a+b}F(a,b)e_{a,b}.
\]
\end{proof}

\begin{remark}\label{rmk:tu-gap-graph}
The change of basis in the proof is the bipartite sign change associated to
the graph $G'_{N-1,2}(TU_N)$ from
Proposition~\ref{prop-TU-bipartite}. In this basis, $L$ is the ordinary graph
Laplacian of $G'_{N-1,2}(TU_N)$ together with a diagonal contribution of
$1$ on each state $e_{a,a+1}$. Under the conventions of
Section~\ref{ssec:twisted-unknot-graphs}, this contribution comes from the
weight $\tfrac12$ loop at each state with adjacent labels.

The functions $\alpha$ and $\beta$ are cosine solutions of the one
dimensional path recurrence. Equation \eqref{eq:cos-endpoints} gives the
endpoint conditions, while \eqref{eq:cos-adjacent} encodes the additional
condition at states with adjacent labels. The latter condition forces
$\theta(N+2)/2=\pi/2$ and accounts for the appearance of $N+2$ in the
spectral gap formula.
\end{remark}

\begin{corollary}\label{cor:tu-polynomial}
The spectral gap of \(TU_N\) in bidegree \((0,3-N)\) admits the asymptotic expansion 
\[
\operatorname{gap}\bigl(\widetilde\Delta_{0,3-N}(TU_N)\bigr)
=
\frac{\pi^2}{(N+2)^2}
-
\frac{\pi^4}{12(N+2)^4}
+
O(N^{-6}).
\]
 Numerical computations through $N\leq 14$ indicate that
this bidegree contains the global minimum gap. Thus, conditional on this
observed pattern, the formula gives the predicted $N^{-2}$ decay of the
global gap. It proves an inverse-polynomial bound in this
specific bidegree.
\end{corollary}

\begin{proof}
Taylor-expand $4\sin^2(\pi/(2(N+2)))$ at $\pi/(2(N+2)) = 0$.
\end{proof}
 
% - - - - - - - - - - - - - - - - - -
% - - - - - - - - - - - - - - - - - -
\subsection{Graphs for \texorpdfstring{$(2,N)$}{(2,N)}-torus knots}\label{ssec:two-n-torus-graphs}
% - - - - - - - - - - - - - - - - - -
Our analysis of twisted unknots $TU_N$ can be extended to $(2,N)$ torus knots.  These are knots that can be obtained as a different closure of the $N$-fold twist 
\[
T_{(2,N)}\;\; := \;\; 
\hackcenter{   \begin{tikzpicture}[scale=.6]
      \draw [ very thick]    (1,1) .. controls +(0,.35) and +(0,-.35) .. (2,2);
  \draw [ very thick]    (1,2) .. controls +(0,.35) and +(0,-.35) .. (2,3);
   \draw [ very thick]    (1,4) .. controls +(0,.35) and +(0,-.35) .. (2,5);
  \path [fill=white] (1.35,1) rectangle (1.65,5);
  \draw [ very thick]    (2,1) .. controls +(0,.35) and +(0,-.35) .. (1,2);
  \draw [ very thick]    (2,2) .. controls +(0,.35) and +(0,-.35) .. (1,3);
    \draw [ very thick]    (2,4) .. controls +(0,.35) and +(0,-.35) .. (1,5);
     \draw [ very thick]    (2,5) .. controls +(0,.5) and +(0,.5) .. (3,5) to (3,1);
     \draw [ very thick]    (2,1) .. controls +(0,-.5) and +(0,-.5) .. (3,1);
      \draw [ very thick]    (1,5) .. controls +(0,1.1) and +(0,1.1) .. (4,5) to (4,1);
     \draw [ very thick]    (1,1) .. controls +(0,-1.1) and +(0,-1.1) .. (4,1);
      \node at (1.5,3.75) {$\vdots$};
    \end{tikzpicture}  } 
    \qquad \qquad 
    \sigma^N \;\; := \;\; 
\hackcenter{   \begin{tikzpicture}[scale=.6]
      \draw [ very thick]    (1,1) .. controls +(0,.35) and +(0,-.35) .. (2,2);
  \draw [ very thick]    (1,2) .. controls +(0,.35) and +(0,-.35) .. (2,3);
   \draw [ very thick]    (1,4) .. controls +(0,.35) and +(0,-.35) .. (2,5);
  \path [fill=white] (1.35,1) rectangle (1.65,5);
  \draw [ very thick]    (2,1) .. controls +(0,.35) and +(0,-.35) .. (1,2);
  \draw [ very thick]    (2,2) .. controls +(0,.35) and +(0,-.35) .. (1,3);
    \draw [ very thick]    (2,4) .. controls +(0,.35) and +(0,-.35) .. (1,5); 
      \node at (1.5,3.75) {$\vdots$};
    \end{tikzpicture}  }
\]
The closure $T_{(2,N)}$ will form a knot when $N$ is odd and a two-strand link when $N$ is even.  For simplicity in our analysis, we consider the knot case when $N$ is odd.  

For $N$ odd, the all zero resolution of $T_{(2,N)}$ will have $N$ circles and $\widetilde\Delta_0$ has the decomposition
\begin{equation}
\widetilde\Delta_0(T_{(2,N)})  = \sum_{i=1}^{N-1} \mathfrak{m}_{i i+1} + \mathfrak{m}_{N1}
\end{equation}
In this case, it is clear that $\Upsilon$ contains a single odd cycle $(1,2)(2,3) \dots (N-1,N)(1,N)$, so that two edges must be removed from $G_{n,k}(T_{(2,N)})$ to make it bipartite.    

Lemma~\ref{lem:loop-count} implies that the sum of the weighted loops in $G_{n,k}(T_{(2,N)})$ are given by $\frac{N}{2}\binom{N-2}{k-2}$. Taking $S=V$, we have $e_{\min}(V)=2+\operatorname{loop}(V)$, since two unit-weight non-loop edges must be removed from $G'_{n,k}(T_{(2,N)})$ to make it bipartite.  Thus,
\begin{equation}
    \Psi_V = \frac{2+\operatorname{loop}(V)}{|V|} = \frac{2 + \frac{N}{2}\binom{N-2}{k-2} }{\binom{N}{k} }
\end{equation}
Theorem~\ref{thm:Desai} then implies that
\[
\lambda_{\min} \leq 4 \Psi_V. 
\]
In certain $q$-degrees one could use a variant of Theorem~\ref{thm:subsets} to establish a lower bound.  But again, it will be challenging to explicitly compute the minimal $\Psi_S$ in all $q$-degrees.

\subsubsection{The exact bidegree spectral gap of \texorpdfstring{$T_{(2,N)}$}{T(2,N)}}\label{sssec:torus-min-gap}
For odd $N \geq 3$, numerical evidence from~\cite{schmidhuber2025quantumalgorithmkhovanovhomology} indicates that the minimum spectral gap for the standard positive-braid diagram occurs in the normalized bidegree $(N-1,3N-2)$. In this $q$-degree, the relevant cochain spaces in homological degrees $N-1$ and $N$ each have dimension $N$. The following theorem computes the Laplacian in homological degree $N-1$ exactly.

\begin{thm}\label{thm:torus-spectral-gap}
For odd $N \geq 3$, the Khovanov Laplacian of the $(2,N)$-torus knot $T_{(2,N)}$ in bidegree $(N - 1,\; 3N - 2)$ has smallest eigenvalue
\[
2 - 2\cos\!\left(\frac{\pi}{N}\right)
\;=\;
4\sin^2\!\left(\frac{\pi}{2N}\right).
\]
Equivalently, since the Laplacian has no kernel in this bidegree, this is the bidegree spectral gap. The asymptotic decay is $\pi^2/N^2+O(N^{-4})$.
\end{thm}

\begin{proof}
Present $T_{(2,N)}$ as the closure of the positive braid $\sigma_1^N$, and index crossings cyclically by $\Z/N\Z$. In the bidegree $(N-1,3N-2)$, the cochain group $C_{\mathrm{KH}}^{N-1,3N-2}(T_{(2,N)})$ has basis
\[
v_i,\qquad i\in \Z/N\Z,
\]
where $v_i$ is the enhanced state with the unique $0$-resolution at crossing $i$ and with all circles labeled by $1$. In the same $q$-degree, $C_{\mathrm{KH}}^{N,3N-2}(T_{(2,N)})$ has basis
\[
w_i,\qquad i\in \Z/N\Z,
\]
where the resolution is the all-one resolution and the unique $x$-label lies on the $i$th circle. There is no contribution from $C_{\mathrm{KH}}^{N-2,3N-2}$ to $C_{\mathrm{KH}}^{N-1,3N-2}$, so the Laplacian on $C_{\mathrm{KH}}^{N-1,3N-2}$ is $d^\dagger d$.

The differential from $v_i$ to the all-one resolution is the local comultiplication at the $i$th crossing. In this extreme $q$-degree, only
\[
\Delta_V(1)=1\otimes x+x\otimes 1
\]
contributes. Thus, for some cube sign $\epsilon_i\in\{\pm1\}$,
\begin{equation}\label{eq:torus-differential}
d(v_i)=\epsilon_i(w_i+w_{i+1}),
\end{equation}
with indices understood modulo $N$. It follows that the matrix of $d^\dagger d$ in the basis $\{v_i\}$ has diagonal entries $2$, off-diagonal entries $\epsilon_i\epsilon_{i+1}$ between cyclic neighbors $i$ and $i+1$, and all other entries zero.

Let $S$ be the diagonal sign matrix defined by $S(v_i)=\epsilon_i v_i$. Since the off-diagonal signs factor as $\epsilon_i\epsilon_{i+1}$, conjugation by $S$ removes all signs. Hence
\[
S(d^\dagger d)S^{-1}=Q_N:=2I+A(C_N),
\]
where $A(C_N)$ is the adjacency matrix of the $N$-cycle and $Q_N$ is its signless Laplacian. Therefore the spectrum of the bidegree Laplacian is the spectrum of $Q_N$. Since $Q_N$ is circulant, its eigenvalues are 
\[
\lambda_k=2+z_k+z_k^{-1}=2+2\cos\left(\frac{2\pi k}{N}\right),
\qquad k=0,1,\ldots,N-1.
\]
Since $N$ is odd, the minimum occurs at $k=(N\pm1)/2$ and equals
\[
2+2\cos\left(\pi-\frac{\pi}{N}\right)
=
2-2\cos\left(\frac{\pi}{N}\right)
=
4\sin^2\left(\frac{\pi}{2N}\right).
\]
There is no zero eigenvalue for odd $N$, so this minimum is the bidegree spectral gap. Taylor expansion of $4\sin^2(\pi/(2N))$ gives the stated asymptotic.
\end{proof}

\begin{remark}
Theorem~\ref{thm:torus-spectral-gap} is a specific bidegree computation. It proves inverse-polynomial decay in the normalized bidegree $(N-1,3N-2)$. Numerical computations through the checked range indicate that this bidegree contains the global minimum gap for the standard positive-braid diagrams of $T_{(2,N)}$, but we do not prove here that this is a global minimum.
\end{remark}

\subsection{Even twist knots}\label{ssec:even-twist-knots}

Combining the methods from Theorems \ref{thm:tu-gap} and \ref{thm:torus-spectral-gap}, we compute the spectral gap of \textit{twisted} knots $\mathrm{Tw}_N$ (see Definition \ref{def:twist-knots}) in bidegree $(N-2, 3N-9)$. Large-scale computations from~\cite{schmidhuber2025quantumalgorithmkhovanovhomology} suggest that the minimum gap over all bidegrees and all alternating knots is achieved in this particular bidegree for knots $\mathrm{Tw}_N$.

\begin{definition}\label{def:twist-knots}
    For even $N \in \mathbb{N}, N \geq 4$, define the twisted knot diagram $\mathrm{Tw}_N$ to be:
    \begin{align*}
\mathrm{Tw}_N \; := \;
\hackcenter{\begin{tikzpicture}[scale=.52]
      \draw [ very thick]    (1,1) .. controls +(0,.35) and +(0,-.35) .. (2,2);
  \draw [ very thick]    (1,2) .. controls +(0,.35) and +(0,-.35) .. (2,3);
   \draw [ very thick]    (1,4) .. controls +(0,.35) and +(0,-.35) .. (2,5);
  \path [fill=white] (1.35,1) rectangle (1.65,5);
  \draw [ very thick]    (2,1) .. controls +(0,.35) and +(0,-.35) .. (1,2);
  \draw [ very thick]    (2,2) .. controls +(0,.35) and +(0,-.35) .. (1,3);
    \draw [ very thick]    (2,4) .. controls +(0,.35) and +(0,-.35) .. (1,5);
     \draw [ very thick]    (2,5) .. controls +(0,.5) and +(0,.5) .. (3,5) to (3,3.5);
  \draw[very thick] (3,2.25) arc[start angle=180, end angle=90, radius=1];
    \path [fill=white] (3,3) rectangle (3.65,2.75);
    \draw[very thick] (3,3.5) arc[start angle=180, end angle=300, radius=1];
    \draw[very thick] (5,2.25) arc[start angle=0, end angle=90, radius=1];
     \draw [ very thick]    (2,1) .. controls +(0,-.5) and +(0,-.5) .. (3,1) -- (3,2.25);
      \draw [ very thick]    (1,5) .. controls +(0,1.4) and +(0,1.4) .. (5,5) to (5,3.1);
     \draw [ very thick]    (1,1) .. controls +(0,-1.4) and +(0,-1.4) .. (5,1) -- (5,2.25);
      \node at (1.5,3.75) {$\vdots$};
    \end{tikzpicture}}
    \end{align*}
    The diagram has $2$ crossings on the right and $N-2$ crossings on the left. 
\end{definition}

As the methods have been developed in detail in the proofs of Theorems \ref{thm:tu-gap} and \ref{thm:torus-spectral-gap}, we omit some of the details from the following proof.

\begin{thm} \label{thm:even-twist-knots}
    The spectral gap of $\mathrm{Tw}_N$ in bidegree $(N-2, 3N-9)$ is
    \begin{align*}
        2-2\cos\left(\frac{\pi}{N -1 }\right) = 4\sin^2\left( \frac{\pi}{2(N -1)}\right).
    \end{align*}
    Since the Laplacian has no kernel in this bidegree, the spectral gap and the smallest eigenvalue agree. Equivalently, the asymptotic expansion in the natural parameter $N-1$ is $\pi^2/(N-1)^2+O((N-1)^{-4})$.
\end{thm}
\begin{proof}
    The normalized homological degree $N-2$ corresponds to the all-one resolution, which has $N-1$ circles and can be depicted as:
    \begin{align*}
        \hackcenter{\begin{tikzpicture}[scale=.52]
  %==== twist region: every crossing gets the 1-smoothing (horizontal caps) ====
  %     => a vertical stack of circles (one between each pair of crossings)
  % bottom cap : part of the LOWER big circle
  \draw[very thick] (1,1) .. controls +(0.25,0.4) and +(-0.25,0.4) .. (2,1);
  % circle at level 2
  \draw[very thick] (1,2) .. controls +(0.25,0.4)  and +(-0.25,0.4)  .. (2,2);
  \draw[very thick] (1,2) .. controls +(0.25,-0.4) and +(-0.25,-0.4) .. (2,2);
  % circle at level 3
  \draw[very thick] (1,3) .. controls +(0.25,0.4)  and +(-0.25,0.4)  .. (2,3);
  \draw[very thick] (1,3) .. controls +(0.25,-0.4) and +(-0.25,-0.4) .. (2,3);
  % the remaining circles of the twist region
  \node at (1.5,4) {$\vdots$};
  % circle at level 5
  \draw[very thick] (1,5) .. controls +(0.25,0.4)  and +(-0.25,0.4)  .. (2,5);
  \draw[very thick] (1,5) .. controls +(0.25,-0.4) and +(-0.25,-0.4) .. (2,5);
  % top cap : part of the UPPER big circle
  \draw[very thick] (1,6) .. controls +(0.25,-0.4) and +(-0.25,-0.4) .. (2,6);
  %==== outer arcs of the diagram (unchanged) ====
  \draw[very thick] (2,6) .. controls +(0,.6)   and +(0,.6)   .. (3,6) -- (3,3.5);   % a
  \draw[very thick] (2,1) .. controls +(0,-.5)  and +(0,-.5)  .. (3,1) -- (3,2.25);  % e
  \draw[very thick] (1,6) .. controls +(0,1.6)  and +(0,1.6)  .. (5,6) -- (5,3.1);   % f
  \draw[very thick] (1,1) .. controls +(0,-1.4) and +(0,-1.4) .. (5,1) -- (5,2.25);  % g
  %==== clasp: the 1-smoothing UNCLASPS it into two non-crossing arcs ====
  \draw[very thick] (3,3.5)  .. controls +(0.9,0.45) and +(-0.9,0.45) .. (5,3.1);   % upper circle
  \draw[very thick] (3,2.25) .. controls +(0.8,0.35) and +(-0.8,0.35) .. (5,2.25);  % lower circle
\end{tikzpicture}}
    \end{align*}
    The resolution of all $1$'s is extremal, therefore we only need to consider the down Laplacian $\partial \partial^\dagger$. To obtain the correct quantum degree, we have to label $N-2$ circles with $1$ and one with $x$. It is straightforward to check that all the maps coming into this resolution are multiplications.

    Using arguments analogous to those from the proofs of Theorems \ref{thm:tu-gap} and \ref{thm:torus-spectral-gap}, and using $S$ to denote the sign matrix, one can check that
    \begin{align*}
S(\partial\partial^\dagger)S^{-1}=2I+A(C_{N-1})+B.
    \end{align*}
    All the matrices are $N-1 \times N-1$ matrices, $B$ is a matrix with a $2 \times 2$ block of $1$'s in the top left corner (corresponding to the two larger circles on the picture) and zeros elsewhere. Writing $e_1,\ldots,e_{N-1}$ for the standard basis, the matrix $B$ is the rank-one positive-semidefinite matrix $B=u u^\top$, where $u=e_1+e_2$.

    We remark that $A(C_{N-1}) + B$ is the adjacency matrix of the $(N-1)$-cycle with one additional parallel edge joining two consecutive vertices, and $2I+A(C_{N-1})+B$ is the signless Laplacian of this multigraph. This is to be compared with $Q_N$ from Theorem \ref{thm:torus-spectral-gap}.

    Since $B$ is positive semidefinite, Weyl monotonicity gives 
        \begin{align*}
    \lambda_k\!\left(S(\partial\partial^\dagger)S^{-1}\right) \geq \lambda_k\!\left(2I+A(C_{N-1})\right),
    \end{align*}
    where the eigenvalues are ordered increasingly. In particular, $S(\partial\partial^\dagger)S^{-1}$ has no zero eigenvalue.

    By the earlier computations, we know that the spectral gap of $2I+A(C_{N-1})$ equals $ 4\sin^2\left( \frac{\pi}{2(N-1)}\right)$. Moreover, those computations show that the corresponding eigen-space $W$ is two dimensional. Hence, there is a nonzero vector $x=(x_1,x_2,\ldots,x_{N-1})\in W$ such that $u^\top x=x_1+x_2=0$. This shows that
    \begin{align*}
        S(\partial\partial^\dagger)S^{-1}x &= (2 I+A(C_{N-1}) + B)x \\
        &= (2 I+A(C_{N-1})) x + u u^\top x \\
        & = 4\sin^2\left( \frac{\pi}{2(N-1)}\right) x.
    \end{align*}
    The Weyl lower bound and this eigenvector together show that the smallest eigenvalue of $S(\partial\partial^\dagger)S^{-1}$ is exactly $4\sin^2\left( \frac{\pi}{2(N-1)}\right)$, completing the proof. 
    
\end{proof}

% ==========================================

\section{Combinatorial torsion}\label{sec:torsion-general}

In this section we develop the combinatorial-torsion machinery used in Section~\ref{sec:torsion-knots} to recover the integral torsion of Khovanov homology from the spectrum of the Khovanov Laplacian. The framework follows Reidemeister~\cite{Reidemeister1935}, Franz~\cite{Franz1935}, Turaev~\cite{Turaev2001}, and Mnev~\cite{mnev2014lecturenotestorsions}; we package the alternating product of pseudodeterminants of combinatorial Laplacians directly as the torsion invariant, in agreement with~\cite[Lemma 3.19]{mnev2014lecturenotestorsions}.

The main statement of this section can be summarized as follows. Given a finite-dimensional free cochain complex $C$ over $\Z$, equipped with the natural inner product on $C \otimes \C$, write $\tau^{\C}(C)$ for the alternating product of Laplacian pseudodeterminants and $\tau^{\Z}(C)$ for the alternating product of orders of the torsion subgroups of $\mathrm{H}^*(C;\Z)$. Theorem~\ref{tauZtauCV} below establishes:
\begin{itemize}
\item \emph{when the rational cohomology of $C$ vanishes,} the two quantities coincide, $\tau^{\Z}(C) = \tau^{\C}(C)$, exhibiting the spectrum as a direct readout of integral torsion;
\item \emph{when the rational cohomology is nontrivial,} they differ by a regulator $V(e,g)$ measuring the volume of an integral basis of $\mathrm{H}^*(C;\Z)$ in any orthonormal basis of $\mathrm{H}^*(C;\C)$, with $\tau^{\Z}(C) = \tau^{\C}(C) / V(e,g)$.
\end{itemize}

We adopt the convention throughout that cohomology is viewed as a subspace of the cochain spaces, with the induced inner product, and that bases of cohomology are formed by cochain vectors. Some of our references work over chain complexes, while other work over cochain complexes. The two are related by Poincaré duality, with corresponding sign adjustments in the exponents that we make explicit when relevant.

\begin{definition}\label{def nonzero det}
Let $M \in \mathrm{Mat}_n(\C)$ be a matrix. We write $\det\nolimits'(M)$ for the product of nonzero eigenvalues of $M$.
\end{definition}

\begin{definition}\label{combinatorialtorsion}
Let $C=(C^i,\partial_i)$ be a finite-dimensional cochain complex of $\C$-vector spaces, with each $C^i$ equipped with an inner product and cohomology identified with the harmonic subspace. Define the spectral $R$-torsion of $C$ by
\begin{equation}\label{eq:Rtor}
\tau^\C(C)
=
\prod_i \det\nolimits'(\Delta_i)^{(-1)^{i+1}i/2}
\in \C^*/\{\pm1\},
\end{equation}
where
\[
\Delta_i=\partial_i^\dagger\partial_i+\partial_{i-1}\partial_{i-1}^\dagger
\]
is the combinatorial Laplacian and $\partial_i^\dagger$ is the adjoint with respect to the chosen inner products. The fact that $\tau^\C(C)$ takes values in $\overline{\mathrm K_1}(\C)=\C^*/\{\pm1\}$ follows from~\cite{mnev2014lecturenotestorsions}.
\end{definition}
The classical definition of torsion (see \cite[Definition 3.3]{mnev2014lecturenotestorsions} or \cite{Turaev2001}) depends on the choice of bases. In principle we should write $\tau^\C(C;e,g)$ for a graded basis $e$ of $C$ and a graded basis $g$ of $\mathrm H^*(C)$. However, when both bases are orthonormal, Definition~\ref{combinatorialtorsion} agrees with $\tau^\C(C;e,g)$ from the classical definition and is independent of the particular orthonormal bases. This follows from~\cite[Lemmas 3.18--3.19]{mnev2014lecturenotestorsions}. In that reference the pseudodeterminant formula appears as a result rather than as a definition. We use it as the definition here in order to keep the discussion centered on combinatorial Laplacians.

The independence of the choice of basis boils down to the change of basis formula for $R$-torsion. The following is adapted from \cite[Definition 3.3]{mnev2014lecturenotestorsions}. 

\begin{definition}
Let $e=(e_i)_i$ and $f=(f_i)_i$ be two ordered bases of a finite-dimensional vector space over $\C$. Write
\[
e_i=\sum_j a_{ij}f_j,
\]
and denote the corresponding change-of-basis matrix by $B_{e,f}=(a_{ij})$.
\end{definition}

\begin{definition}\label{def:change basis in torsion}
Let $C=(C^i,\partial_i)$ be a finite-dimensional cochain complex of $\C$-vector spaces. Let $f^i$ be a basis of $C^i$ and $h^i$ a basis of $\mathrm H^i(C)$. Choose orthonormal bases $e^i$ of $C^i$ and $g^i$ of $\mathrm H^i(C)$. We define the $R$-torsion with respect to $f=(f^i)_i$ and $h=(h^i)_i$ by
\begin{equation}\label{eq:Rtorbasis}
\tau^\C(C;f,h)
=
\tau^\C(C)
\prod_i
\det(B_{e^i,f^i})^{(-1)^i}
\det(B_{g^i,h^i})^{(-1)^{i+1}}.
\end{equation}
Here $\tau^\C(C)$ is the spectral torsion from Definition~\ref{combinatorialtorsion}.
\end{definition}

Recall that the determinants of change of basis matrices between orthonormal bases are all equal to $\pm 1$ and $\tau^\C(C)$ takes values in $\C^*/\{\pm1\}$. It follows that, when $f$ and $h$ are orthonormal, we have $\tau^\C(C;f,h) = \tau^\C(C;e,g)$. Hence, the above definition really does not depend on the choice of the bases $e,g$. We will use the formula \eqref{eq:Rtorbasis} in what follows in order to adapt the results about $R$-torsion to Definition \ref{combinatorialtorsion} built on combinatorial Laplacians.

Our goal is to explain how the quantity in \eqref{eq:Rtor} extracts information relating to the integral torsion in $\mathrm{H}(C,\Z)$. We package the homological torsion into a single quantity in the next definition. Compare this with \cite[Section 4]{Turaev2001}.

\begin{definition}\label{deford}
    Let $C = (C^i, d_i)$ be a finite dimensional free cochain complex over $\Z$ with cohomology $\mathrm{H}(C) = \mathrm{H}(C, \Z)$. We define the integral ($R$-)torsion as 
\begin{equation}
            \tau^{\Z}(C)= \prod_{i} \left(\mathrm{ord}\left(\mathrm{Torsion}\left(\mathrm{H}^{i}(C) \right) \right) \right)^{(-1)^{i+1}}.
\end{equation}
Here $\mathrm{ord}(\mathrm{Torsion}(\mathrm{H}^{i}(C)))$ is the order (number of elements) of the torsion part of $\mathrm{H}^i(C)$, the $i$-th cohomology group of $C$. The order is to be understood in the group theoretic sense as the number of elements.
\end{definition}
When $\mathrm{H}(C, \Z)$ has no free summand, one could interpret the contribution from the free part of $\mathrm{H}(C, \Z)$ as a factor of $1$ in the expression for $\tau^{\Z}(C)$. Since we can infer all the information about the free part of $\mathrm{H}(C, \Z)$ from the cohomology over $\C$, we do not need to encode that information into $\tau^\Z$ itself.

In order to relate the integral $R$-torsion of some (finite dimensional, free) cochain complex $C$ over $\Z$ to the $R$-torsion of $C \otimes \C$ we need an auxiliary quantity, called the regulator or the volume associated to a basis of $C \otimes \C$ .

\begin{definition}\label{defvolume}
    Let $C = (C^j, \partial_j)_j$ over $\C$ be a cochain complex with inner product $\langle\_, \_\rangle$. We pick a basis $e = (e_{ij})_{ij}$ of $C$ such that $e_j = (e_{ij})_i$ is a basis for $C^j$ for each fixed $j$. For any orthonormal basis $f = (f_{ij})_{ij}$ of $C$ such that $f_j = (f_{ij})_i$ is basis for $C^j$ for each fixed $j$, we define the volume of $e$ as
    \begin{align*}
        \widetilde v_j(e, f) &= \left| \begin{pmatrix}
            \langle e_{1j}, f_{1j}  \rangle & \langle e_{1j}, f_{2j}  \rangle & \cdots & \langle e_{1j}, f_{mj}  \rangle \\
            \langle e_{2j}, f_{1j}  \rangle & \langle e_{2j}, f_{2j}  \rangle & \cdots & \langle e_{2j}, f_{mj}  \rangle \\
             \vdots & \vdots & \ddots & \vdots\\
             \langle e_{mj}, f_{1j}  \rangle & \langle e_{mj}, f_{2j}  \rangle & \cdots & \langle e_{mj}, f_{mj}  \rangle
        \end{pmatrix}\right|\\
        V(e,f) &= \prod_{j} \widetilde v_j(e,f)^{(-1)^{j+1}}
    \end{align*} 
\end{definition}
    Note that the above does not depend on the choice of orthonormal basis. While we could suppress it from the definition we leave it there for clarity. We could generalize the following theorem by allowing different bases and including change of bases terms. For clarity we omit this generalization.

\begin{thm}\label{tauZtauCV}
    Let $C = (C^j, \partial_j)_j$ be a finite dimensional free cochain complex over $\Z$. Assume that the canonical integral basis of $C$ is orthonormal in the inner product of $C\otimes \C$. The cohomology of $C$ over $\Z$ has an integral basis for the free part $e = (e_i)_i$. For any orthonormal basis $g = (g_i)_i$ of the cohomology of $C \otimes \C$,
    \begin{align*}
        \tau^\Z(C) &= \tau^\C(C)/V(e,g).
    \end{align*}
\end{thm}
\begin{proof}
    This is precisely the statement of \cite[Equation 6]{L2013}. 
\end{proof}

Torsion invariants are often used to study chain homotopy equivalences or quasi-isomorphisms; see \cite{CHUNG2009, mnev2014lecturenotestorsions}. To look at torsion of maps we need to know how to compute torsion of a short exact sequence, we may then use the cone construction to compute the torsion of a quasi-isomorphism.

\begin{lemma}\label{SEStorsion}
Let $A,B,C$ be finite dimensional free cochain complexes over $\Z$ or $\C$ that fit into a short exact sequence of complexes as shown.
            \begin{figure}[ht]
        \centering
        \begin{tikzcd}
        0 \arrow[r] & A \arrow[r, tail] & C \arrow[r, two heads] & B \arrow[r] & 0
        \end{tikzcd}
        \end{figure}
        Then 
        \begin{align*}
            \tau^\C(C) = \tau^\C(A)\tau^\C(B)\tau^\C(\chi),
        \end{align*} where $\chi$ is the long exact sequence of cohomologies induced by the displayed short exact sequence, see \cite[\href{https://stacks.math.columbia.edu/tag/0117}{Tag 0117}]{stacks-project}.
\end{lemma}
\begin{proof}
    This is \cite[lemma 3.12]{mnev2014lecturenotestorsions}.
\end{proof}

\begin{definition}\label{deftorsionofmap}
    Let $A,B$ be finite dimensional free cochain complexes over $\Z$ or $\C$. Suppose $f:A\to B$ is a quasi-isomorphism. Then we define
    \begin{align*}
        \tau(f) &= \tau(\mathrm{cone}(f)),
    \end{align*}
    where $\mathrm{cone}(f)$ is the mapping cone, see  \cite[\href{https://stacks.math.columbia.edu/tag/014D}{Tag 014D}]{stacks-project}. We understand the torsion of quasi-isomorphism to be computable using Lemma \ref{SEStorsion}. 
\end{definition}

\section{Torsion of Khovanov homology}\label{sec:torsion-knots}

We now specialize the combinatorial-torsion framework of Section~\ref{sec:torsion-general} to the Khovanov complex of a knot diagram. The main result is Theorem~\ref{thm:integral-torsion}. In $q$-degrees where the rational Khovanov cohomology vanishes, the higher Khovanov spectrum recovers the order of the integral torsion subgroup of $\mathrm{KH}^*(K;\Z)$ exactly. In $q$-degrees where the rational cohomology is nonzero, the two quantities differ by an explicit regulator computed from the integral and orthonormal bases of $\mathrm{KH}^*(K;\C)$. This proves a conjecture left implicit in the work of Ortiz-Navarro~\cite[Section 5]{ORTIZNAVARRO2012}, see Remark \ref{rmk:ortiz_navarro_conj}. 

The section concludes with two Examples \ref{eg:trefoil torsion} and \ref{exam:twisted} which illustrate the two cases. The trefoil $3_1$ in $q$-degree $-7$ has trivial rational Khovanov cohomology, and the spectrum directly detects the $\Z/2$ torsion in $\mathrm{KH}^*(3_1;\Z)$ via $\tau^{\C} = \tau^{\Z} = 1/2$. The once-twisted unknot formally shifted to homological degree $n$ has nontrivial rational cohomology, and the regulator contributes a factor of $\sqrt{2}^{(-1)^{n+1}}$ in homological degree $n$.

Khovanov homology is well defined over $\Z$. For a diagram $D$ of $K$, we move freely between $C_{\mathrm{KH}}(D;\Z)$, $\mathrm{KH}(K;\Z)$, and their complexifications $C_{\mathrm{KH}}(D;\C)=C_{\mathrm{KH}}(D;\Z)\otimes\C$ and $\mathrm{KH}(K;\C)$ as convenient.

\begin{definition}\label{khovanovtorsion}
Let $D$ be a diagram of a knot $K$. For each normalized $q$-degree $j$, define the diagram-dependent spectral torsion and the integral Khovanov torsion by
\begin{align}
\tau_j^\C(D)&:=\tau^\C\bigl(C_{\mathrm{KH}}^{*,j}(D;\C)\bigr),\\
\tau_j^\Z(K)&:=\tau^\Z\bigl(C_{\mathrm{KH}}^{*,j}(D;\Z)\bigr).
\end{align}
For the total complex, set
\begin{align}
\tau^\C(D)&:=\tau^\C\bigl(C_{\mathrm{KH}}(D;\C)\bigr),\\
\tau^\Z(K)&:=\tau^\Z\bigl(C_{\mathrm{KH}}(D;\Z)\bigr).
\end{align}
The integral quantities are link invariants. The spectral quantities depend in general on the chosen diagram and on the orthonormal enhanced-state basis. We use the same distinction for the basis-dependent torsions $\tau^\C(D;e,f)$ and $\tau_j^\C(D;e,f)$.
\end{definition}

\begin{lemma}
    The Definition \ref{khovanovtorsion} agrees with the definition of Khovanov volume form of \cite{ORTIZNAVARRO2012} evaluated in orthonormal bases for the chain complex and its homology. The particular choice of the orthonormal bases does not matter by Definition \ref{combinatorialtorsion} and the discussion following that definition.
\end{lemma}
\begin{proof}
    The basis of $V^{\otimes n}$ in \cite[Section 3.2]{ORTIZNAVARRO2012} is orthonormal with respect to our inner product on $V$. Therefore the result follows from \cite[Lemma 3.18, Lemma 3.19]{mnev2014lecturenotestorsions}.
\end{proof}
The integral basis for Khovanov cochain complexes that we are referring to is the basis where $1 < x$, extended to tensor products using the reverse lexicographical ordering. Clearly this basis is orthonormal in our chosen inner product. For the basis of integral Khovanov homology we use the integral basis obtained using the Smith normal form of the differentials. Since any integral invertible matrix must have determinant $\pm 1$, it follows from Definition \ref{def:change basis in torsion} that a particular choice of integral basis does not change the content of the next theorem. These are the bases described in \cite[Section 3.2, Section 4.4.]{ORTIZNAVARRO2012} and we refer the reader there to see the bases spelled out explicitly. 

\begin{thm}\label{thm:integral-torsion}
    Denote the basis of $C_{\mathrm{KH}}$ from \cite[Section 3.2]{ORTIZNAVARRO2012} as $e = (e_i)_i$ and the (induced) basis of cohomology from \cite[Section 4.4.]{ORTIZNAVARRO2012} as $f = (f_j)_j$. Then 
    \begin{align*}
        \tau^\C(D;e,f) &= \tau^\Z(K), \\
        \tau_j^\C(D;e,f) &= \tau_j^\Z(K).
    \end{align*}
\end{thm}
\begin{proof}
    This is immediate from the cited reference and Theorem \ref{tauZtauCV}. The second equality follows from the first as the Khovanov cochain complex splits as a direct sum over the $q$-degrees $j$.
\end{proof}

This theorem proves the conjecture that was left implicit in \cite[Section 5]{ORTIZNAVARRO2012}.

\begin{cor}
    The torsions $\tau_j^\C(D;e,f)$ and $\tau^\C(D;e,f)$ evaluated in the canonical integral bases described above are knot invariants. In particular, they do not depend on the choice of a knot diagram.
\end{cor}
\begin{proof}
    This follows from the previous theorem, the fact that Khovanov cohomology is a knot invariant, and from the definition of integral $R$-torsions. This was previously proven in \cite{ORTIZNAVARRO2012}.
\end{proof}

The distinction between integral and orthonormal bases is essential. A Smith normal form computation of the integer differentials produces a basis for the saturated integral cohomology lattice, allowing one to evaluate its volume against an orthonormal harmonic basis. The same computation also determines the integral Khovanov homology itself. Thus this procedure is not an algorithmic shortcut for computing integral torsion. 

\begin{cor}\label{cor:khov-vol-snf}
Let \(e=(e_i)_i\) denote the integral basis of
\(C_{\mathrm{KH}}(D;\mathbb Z)\) from~\cite[Section 3.2]{ORTIZNAVARRO2012},
viewed as an orthonormal basis after extending scalars to \(\mathbb C\).
Let \(f\) be an orthonormal basis of
\(\mathrm{KH}^*(K;\mathbb C)\). Let \(g\) be a basis for the free part of the
saturated integral lattice
\[
\mathrm{KH}^*(K;\mathbb Z)/\operatorname{Tors}
\hookrightarrow
\mathrm{KH}^*(K;\mathbb Q),
\]
computed from the integer differentials of
\(C_{\mathrm{KH}}(D;\mathbb Z)\) via Smith normal form. Then
\[
\tau^\C(D)
=
\tau^\Z(K)\,V(g,f),
\]
and, in each \(q\)-degree \(j\),
\[
\tau_j^\C(D)
=
\tau_j^\Z(K)\,V_j(g,f),
\]
where \(V_j(g,f)\) denotes the volume of the \(j\)-th \(q\)-graded part of
the integral cohomology lattice measured against the orthonormal basis \(f\).

Equivalently, the basis-dependent torsion satisfies
\[
\tau^\C(D;e,g)=\tau^\Z(K),
\qquad
\tau_j^\C(D;e,g)=\tau_j^\Z(K),
\]
and the displayed formulas above are obtained by changing the cohomology
basis from \(g\) to the orthonormal basis \(f\).
\end{cor}

\begin{proof}
Since
\(C_{\mathrm{KH}}(D;\mathbb Z)\) is a finite free cochain complex over
\(\mathbb Z\), reducing its integer differentials to Smith normal form
computes both the torsion subgroups of integral Khovanov cohomology and a
basis \(g\) for the free part of the saturated integral lattice
\[
\mathrm{KH}^*(K;\mathbb Z)/\operatorname{Tors}
\subset
\mathrm{KH}^*(K;\mathbb Q).
\]
Theorem~\ref{thm:integral-torsion} identifies the basis-dependent
complex torsion evaluated in the chain basis \(e\) and the
integral cohomology basis \(g\) with the integral torsion:
\[
\tau^\C(D;e,g)=\tau^\Z(K),
\qquad
\tau_j^\C(D;e,g)=\tau_j^\Z(K).
\]
Changing the cohomology basis from \(g\) to the orthonormal basis \(f\)
introduces exactly the regulator factor \(V(g,f)\), and similarly
\(V_j(g,f)\) in fixed \(q\)-degree. The result follows from 
Theorem~\ref{tauZtauCV}.
\end{proof}

\begin{remark}\label{rmk:ortiz_navarro_conj}
The same argument applies to any finite free cochain complex over
\(\mathbb Z\) equipped with an inner product after extension of scalars to
\(\mathbb C\). 

The statement about $\tau_j^\C(D;e,g)$ from the above corollary, provides an explanation for the phenomena observed in~\cite[Section 5.4]{ORTIZNAVARRO2012} and shows that Reidemeister torsion detects algebraic torsion in Khovanov homology for any knot or link, generalizing the statement from the abstract of \cite{ORTIZNAVARRO2012}.
\end{remark}

The next two examples are illustrative and demonstrate the content of the previous statements. First, we look at the simplest example of integral torsion in Khovanov cohomology for knots. That is the $\Z/2\Z$ torsion that appears in the computations for the trefoil.

\begin{example}\label{eg:trefoil torsion}
    Let $D_{3_1}$ be the standard diagram of the trefoil $3_1$. We compute its torsion in $q$-degree $-7$. Restricting \cite[Figure 3]{Bar_Natan_2005} to $q$-degree $-7$ gives us the following cochain complex.
      \begin{figure}[H]
        \centering
\begin{tikzcd}
            &                                                                                              & \langle xx \rangle[-2] \arrow[rd] &   \\
0 \arrow[r] & {\langle 1xx, x1x, xx1\rangle}[-3] \arrow[ru, "m_{13}"] \arrow[r, "m_{23}"] \arrow[rd, "m_{12}"] & \langle xx \rangle[-2] \arrow[r]  & 0 \\
            &                                                                                              & \langle xx \rangle [-2]\arrow[ru] &  
\end{tikzcd}
    \end{figure}
\noindent Here $[i]$ denotes the $i$-th homological degree. We will compute the Laplacians for cohomological degrees $-3$ and $-2$. First, we observe that: 
\begin{align*}
    m^\dagger: &1\mapsto 11 \\
    &x \mapsto 1x +x1, \\
    mm^\dagger:& 1 \mapsto 1 \\
    &x \mapsto 2x, \\
    m^\dagger m:& 11 \mapsto 11 \\
    &1x \mapsto 1x + x1 \\
    &x1 \mapsto 1x + x1 \\
    &xx \mapsto 0.
\end{align*}
That gives us the Laplacians:
    \begin{align*}
        \Delta_{-3}^{\mathrm{KH}} = \Delta_{-2}^{\mathrm{KH}} &=  
\begin{pmatrix}
2 & 1 & 1 \\
1 & 2 & 1 \\
1 & 1 & 2
\end{pmatrix} 
    \end{align*}
    The eigenvalues of this matrix are $1,1,4$. Hence: 
    \begin{align*}
        \tau^\C(C_{\mathrm{KH}}^{*,-7}(D_{3_1})) &= \prod_{i\in \{-3,-2\}}4^{(-1)^{i+1} i/2} \\
        &= \frac{1}{8} \cdot 4 = \frac{1}{2} \\
        \tau^{\Z}(C_{\mathrm{KH}}^{*,-7}(D_{3_1})) &= 2^{(-1)^{-1}} = \frac{1}{2}
    \end{align*}
where the last line follows from the fact that the integral Khovanov cohomology in $q$-degree $-7$ only has one nontrivial group $\Z/2\Z$ in homological degree $-2$ (see \cite{Bar_Natan_2005}). 
\end{example}
In the second example we exhibit the presence of nontrivial volume. 

\begin{example} \label{exam:twisted}
    We will compute the $R$-torsion of the once twisted unknot, shifted to homological degree $n$ (the degree shift can be interpreted as if the unknot was a component of a link, disjoint from the rest of the link). Khovanov cube looks like:
    
    \begin{figure}[ht]
        \centering
\begin{tikzcd}
{\langle 1,x \rangle [n]} &  &                             \\
\oplus \arrow[rr, "m"]    &  & {\langle 1,x \rangle [n+1]} \\
{\langle 1,x \rangle [n]} &  &                            
\end{tikzcd}
    \end{figure}
\noindent As in the previous example compute the Laplacian:
    \begin{align*}
        \Delta_{n} &=  
\begin{pmatrix}
1 & 0 & 0 & 0\\
0 & 1 & 1 & 0\\
0 & 1 & 1 & 0 \\
0 & 0 & 0 & 0
\end{pmatrix} 
    \end{align*}
    The eigenvalues of this matrix are $1,2$. Similar computation shows that the eigenvalues of $\Delta_{n+1}$ are also $1,2$ since
     \begin{align*}
        \Delta_{n+1} &=  
\begin{pmatrix}
1 & 0\\
0 & 2 \\
\end{pmatrix} 
    \end{align*}   
 Hence we can compute 
    \begin{align*}
        \tau^\C(C_{\mathrm{KH}}(TU_1)) &= \prod_{i\in \{n, n+1\}} 2^{(-1)^{i} i/2} \\
        &=2^{(-1)^{n} (n/2 - (n+1)/2)} \\
        & = \sqrt{2}^{(-1)^{n+1}}
    \end{align*}

An integral basis for $\ker(\Delta_n)$ is $1x - x1, xx$ and $\ker(\Delta_{n+1}) = \{\}$. The corresponding orthonormal basis would be $(1x - x1)/\sqrt{2}, xx$. Therefore, the volume of the integral basis is 
     \begin{align*}
        &\det \begin{pmatrix}
\langle 1x - x1 , (1x - x1)/\sqrt{2}\rangle & \langle 1x - x1 , xx\rangle\\
\langle  xx , (1x - x1)/\sqrt{2}\rangle &\langle xx , xx\rangle \\
\end{pmatrix} \\
&=\det \begin{pmatrix}
    \sqrt{2} & 0 \\
    0 & 1
\end{pmatrix} \\
&= \sqrt{2}
    \end{align*} 

We know that this complex is torsion free (it is just the Khovanov homology of the unknot), thus
\begin{align*}
    \sqrt{2} ^{(-1)^{n+1}} \cdot 1 = \sqrt{2}^{(-1)^{n+1}}
\end{align*}
as expected.
\end{example}
When the complex cohomology vanishes, the regulator is trivial and the spectral torsion detects the integral torsion directly. When a free summand is present, additional arithmetic lattice data is genuinely required. The following degreewise reformulation isolates the spectral contribution in a single cohomological degree.

\begin{proposition}\label{propdecomposition cochain}
    Let $C = (C^i, \partial_i)_i$ be a free, finite dimensional cochain complex over $\Z$. Let $E_i$ be the following complex (assigned to $C^i$) in degrees $i-1, i, i+1$: 
    \begin{figure}[ht]
        \centering
        \begin{tikzcd}
0 \arrow[r] & \mathrm{coim}(\partial_{i-1}) \arrow[r, "\partial_{i-1}"] & C^i \arrow[r, "\partial_i"] & \Im(\partial_i) \arrow[r] & 0,
\end{tikzcd}
    \end{figure}
    
\noindent where $\partial_{i-1}$ is restricted to a subdomain. It then holds that:
    \begin{align}
        \mathrm{H}^i(C, \Z) \cong \mathrm{H}(E_i, \Z) = \mathrm{H}^i(E_i, \Z), 
    \end{align}
that is, the cohomology of $E_i$ over $\Z$ vanishes outside degree $i$, where it agrees with $\mathrm H^i(C;\Z)$.
\end{proposition}
\begin{proof}
    This follows from the definition of cohomology. See also \cite[\href{https://stacks.math.columbia.edu/tag/0118}{Tag 0118}]{stacks-project}.
\end{proof}
\begin{thm}
    We use the notation of the previous proposition. Let $e = (e_j)_j$ be an integral basis of $\mathrm{H}^i(C, \Z)$. Equip $C \otimes \C$ with an inner product that makes $e$ orthonormal. The inner product restricts to $E_i$. Let $f = (f_j)_j$ be an orthonormal basis of $\mathrm{H}^i(C, \C)$.  Then it follows that: 
    \begin{align*}
        \tau^\C(E_i \otimes \C) = \left(\mathrm{ord}\left(\mathrm{Torsion}\left(\mathrm{H}^{i}(C) \right) \right) \right)^{(-1)^{i+1}} \widetilde v_i(e,f)^{(-1)^{i+1}},
    \end{align*}
    where we used the notation from Definitions \ref{deford} and \ref{defvolume}. Moreover, the following holds:
    \begin{align*}
        \tau^\C(E_i \otimes \C) =  \det \nolimits' (\partial^\dagger_{i-1}\partial_{i-1})^{(-1)^{i} (i-1)/2} \det \nolimits'(\Delta_{i})^{(-1)^{i+1} i/2} \det \nolimits'(\partial_{i}\partial_{i}^\dagger)^{(-1)^{i+2} (i+1)/2}.
    \end{align*}
\end{thm}
\begin{proof}
    The first equality follows from Proposition \ref{propdecomposition cochain} and the relevant definitions.

    For the second equality, we will only focus on the first nontrivial term of $E_i \otimes \C$. An analogous argument works for the other two factors in the above product. 
    
    We show that the first nontrivial factor of $E_i\otimes\C$, namely $\mathrm{coim}(\partial_{i-1})$, has Hodge-Laplacian spectrum equal to the spectrum of
    \[
    \widehat\Delta_{i-1}
    :=
    \partial_{i-1}^\dagger\partial_{i-1}\big|_{\ker(\partial_{i-1})^\perp}.
    \]
    After extension of scalars to $\C$, the coimage $C^{i-1}/\ker(\partial_{i-1})$ is represented canonically by the orthogonal complement $\ker(\partial_{i-1})^\perp=\Im(\partial_{i-1}^\dagger)$. On this subspace, $\widehat\Delta_{i-1}$ is positive definite. Replacing the quotient by this orthogonal representative changes neither the nonzero spectrum nor its pseudodeterminant. Since the incoming map to the first nonzero term of $E_i$ is zero, its Hodge Laplacian is precisely this restriction of $\partial_{i-1}^\dagger\partial_{i-1}$. The same argument applies to the last term, giving the displayed product.
\end{proof}
\begin{remark}
    Note that the choice of the inner product could be made laxer. However, that would lead to additional terms and factors involving change of bases matrices, which we chose to omit for clarity.
\end{remark}

We would hope that torsion would distinguish some knots that integral Khovanov cohomology does not. While this is not the case when we compute the torsion of $C_{\mathrm{KH}}$, we might be able to differentiate more knots by computing the torsion of chain maps between their corresponding cochain complexes as in the case of lens spaces. Unfortunately this seems relatively far fetched as most knot cobordisms have saddles, caps or cups. These assure that the induced chain map on $C_{\mathrm{KH}}$ is not a quasi-isomorphism for dimension reasons. However, we can say something about the torsion of quasi-isomorphisms. In particular, we know that the torsion of the maps induced by Reidemeister moves equals $1$ (see \cite{ORTIZNAVARRO2012}). This suggests that there are some parallels with simple homotopies, which are used to differentiate lens spaces. See \cite{mnev2014lecturenotestorsions, Turaev2001} for a reference on this classical application of torsion. 

\begin{proposition}\label{prop:torsion of reidemeister moves}
    Let $D_A, D_B$ be two different knot diagrams. Denote by $\rho$ a quasi-isomorphism between $A=C_{\mathrm{KH}}(D_A;\Z)$ and $B=C_{\mathrm{KH}}(D_B;\Z)$. Fix the choices of (integral) basis of the chain complexes $A \otimes \C$ and $B \otimes \C$, and their cohomologies as in \cite[Section 3.2, Section 4.4]{ORTIZNAVARRO2012}. Then 
    \begin{align*}
        \tau^{\C}(\rho) = \tau^{\C}(\chi), 
    \end{align*} where $\chi$ is taken from Lemma \ref{SEStorsion}.
    
    If $D_A$ and $D_B$ represent the same knot then $\tau^{\C}(\rho) = 1$. The same holds for torsions in a specific $q$-degree.
\end{proposition}
\begin{proof}
    The first part of the statement is an application of Lemma \ref{SEStorsion} and Theorem \ref{thm:integral-torsion}. The second statement is the content of \cite[Theorem 2]{ORTIZNAVARRO2012}. 
\end{proof}

\begin{remark}\label{alternative torsion is an invariant proof}
    While the fact that chain maps induced by Reidemeister moves have torsion $1$ is shown in \cite[Theorem 2]{ORTIZNAVARRO2012}, their computations are arduous, but necessary for the proof of invariance of the \textit{volume form} (see \cite{ORTIZNAVARRO2012} for definition and discussion). If one is only interested in the torsion invariants with respect to the chosen basis of the previous proposition, we could use that Proposition \ref{prop:torsion of reidemeister moves} and Theorem \ref{thm:integral-torsion} as a proof of invariance of torsion (in integral bases) under Reidemeister moves. 
\end{remark}

\section{Lee's homology and the Lee spectral sequence}\label{sec:lee}

In~\cite{Lee2005} Lee deformed the Khovanov differential to produce a chain complex whose homology (Lee homology) is isomorphic to $\C^{2^{|K|}}$ (where $|K|$ is the number of components of a link $K$). The filtration associated to the deformation of the differential gives rise to the Lee spectral sequence~\cite{Rasmussen2010}. In this section we use the homological perturbation framework of Section~\ref{subsec:Hom-pert} to give a harmonic model for the Lee spectral sequence. We demonstrate that Lee's canonical generators are not generally harmonic for the Khovanov Laplacian. Instead, the Lee perturbation can be used to transfer the Khovanov harmonic subspace to the Lee chain complex via the homological pertrubation lemma (see Theorem \ref{thm:lee-hodge-transfer}).

\begin{definition}\label{lee}
For a field $F$ and an oriented knot $K$ with diagram $D$, write $C_{\mathrm{Lee}}(D;F)$ for the Lee chain complex and $\mathrm{Lee}(K;F)$ for its homology, in the conventions of~\cite{Lee2005, Rasmussen2010}.  
\end{definition}

\begin{proposition}
\label{prop:lee-generators-not-kh-harmonic}
Lee's canonical chain representatives are not, in general, harmonic for the Khovanov Laplacian.
\end{proposition}

\begin{proof}
Consider the one-crossing diagram of the unknot (see Section \ref{ssec:twisted-unknot-graphs} and Example \ref{exam:twisted}). Its relevant two-term Khovanov chain complex is
\[
0 \longrightarrow V\otimes V \xrightarrow{m} V \longrightarrow 0.
\]
Let
\[
a=1+x,\qquad b=1-x.
\]
When interpreted as a Lee chain complex the multiplication becomes
\[
m_{\mathrm{Lee}}=m+m_\Phi,
\]
where \(m_\Phi(x\otimes x)=1\) and \(m_\Phi\) vanishes on
\(1\otimes 1\), \(1\otimes x\), and \(x\otimes 1\). Hence
\[
m_{\mathrm{Lee}}(a\otimes b)
=
m_{\mathrm{Lee}}\bigl((1+x)\otimes(1-x)\bigr)
=
0.
\]
However, with the ordinary Khovanov differential,
\[
\partial_{\mathrm{KH}}(a\otimes b)
=
m\bigl((1+x)\otimes(1-x)\bigr)
=
1\neq 0.
\]
Therefore \(a\otimes b\) is a Lee cycle but not a Khovanov cycle, and hence cannot be
a harmonic representative for Khovanov homology. 
\end{proof}

 The next theorem follows from the homological perturbation lemma, see for instance \cite{Romero2026}. 

\begin{thm}[Hodge transfer of Lee's perturbation]\label{thm:lee-hodge-transfer}
To more easily distinguish between Khovanov and Lee settings, we modify the notation from Sections \ref{sec:hodge} and \ref{sec:khovanov-background}, let
\[
C=C_{\mathrm{KH}}(D;\mathbb C),\qquad d=\partial_{\mathrm{KH}},
\qquad \Phi=\partial_{\mathrm{Lee}}-\partial_{\mathrm{KH}}, \qquad
L_{\mathrm{KH}}=dd^\dagger+d^\dagger d, \qquad
\mathcal H=\ker L_{\mathrm{KH}}.
\]
Let \(P:C\to \mathcal H\) be the orthogonal projection, let
\(\iota:\mathcal H\hookrightarrow C\) be the inclusion, and let \(G\) be the Green
operator, defined by inverting \(L_{\mathrm{KH}}\) on \(\mathcal H^\perp\) and setting
\(G=0\) on \(\mathcal H\). Set
\[
h=d^\dagger G.
\]
Then
\[
dh+hd=\mathrm{id}_C-\iota P,\qquad P\iota=\mathrm{id}_{\mathcal H},
\]
so \((P,\iota,h)\) is a contraction of \((C,d)\) onto \((\mathcal H,0)\).

The homological perturbation lemma transfers the Lee differential
\(d+\Phi\) to a filtered differential
\[
D_{\mathcal H}
=
P\Phi\sum_{r\geq 0}(-1)^r(h\Phi)^r\iota
\]
on \(\mathcal H\). The resulting complex \((\mathcal H,D_{\mathcal H})\) is chain homotopy equivalent to
\((C,\partial_{\mathrm{Lee}})\). In particular,
\[
\mathrm{H}^*(\mathcal H,D_{\mathcal H})\cong \mathrm{Lee}(K;\mathbb C).
\]
Moreover, the Lee spectral sequence may be computed on the harmonic
Khovanov representatives \(\mathcal H=\ker L_{\mathrm{KH}}\). If a class
\(\alpha\in \mathcal H\) survives to the page on which the term with \(r\) factors of
\(\Phi\) appears, then the corresponding page differential is represented by
\[
d_r(\alpha)
=
(-1)^{r-1}P\Phi(h\Phi)^{r-1}\iota(\alpha),
\]
where the displayed sign follows from the perturbation convention in the series defining $D_{\mathcal H}$. 
\end{thm}

\begin{proof}
The fact that \((P,\iota,h)\) is a contraction of the Khovanov complex
\((C,d)\) onto its harmonic subspace \((\mathcal H,0)\) follows from the results in Section \ref{subsec:Hom-pert}.

Since \(\partial_{\mathrm{Lee}}=d+\Phi\) is a differential and
\(\Phi\) strictly increases the Lee filtration, the perturbation is nilpotent
since the filtered complex is finite. Hence we can apply the homological perturbation lemma and transfer \(d+\Phi\) to the differential
\[
D_{\mathcal H}
=
P\Phi\sum_{r\geq 0}(-1)^r(h\Phi)^r\iota
\]
on \(\mathcal H\), producing a chain homotopy equivalence
\[
(\mathcal H,D_{\mathcal H})\simeq (C,\partial_{\mathrm{Lee}}).
\]
The associated filtered spectral sequence is therefore equivalent to the Lee spectral sequence.
\end{proof}

 \begin{remark}
    This approach generalizes to any spectral sequence originating from a perturbation to the Khovanov differential over $\Q$. This is a very narrow condition. For instance, this approach is not applicable to the Szabó geometric spectral sequence~\cite{https://doi.org/10.1112/jtopol/jtv027} as that spectral sequence is only defined over $\Z/2\Z$.   \end{remark} 

\section{Concluding remarks and future directions}\label{sec:future}

The results of this paper raise several questions that we have not addressed.
The bidegree computations of Section~\ref{sec:graph} identify the smallest eigenvalue in the bidegrees where it is consistently observed numerically. Proving, or characterizing precisely, when these observed bidegrees contain the global minimum of the Khovanov Laplacian for a given diagram remains open. More generally, the signless graph model yields inverse-polynomial lower bounds in fixed-width bands near the extremal $q$-degrees when the corresponding rational Khovanov cohomology vanishes. In bulk $q$-degrees the graph may have exponentially many vertices, so an extension there would require stronger structural estimates for the bipartiteness parameter $\Psi$, or polynomial control of the relevant connected components, rather than a sharper bound on the total vertex count alone.

A related question concerns the behavior of the extreme-degree graphs under local diagram modifications. In favorable situations, Reidemeister moves and other local changes may correspond to controlled graph operations such as vertex insertion, edge insertion, contraction, or loop-weight changes. A graph-level description of these operations would clarify how diagrammatic geometry affects the non-harmonic spectrum. We do not expect such a description, by itself, to imply monotonicity of spectral gaps under Reidemeister moves. The identification of harmonic representatives as $\pm 1$-vertex two-colorings (Corollary~\ref{cor:harmonic-2-coloring}) suggests a parallel question on the torsion side of determining which combinatorial features of the associated graph determine the integer torsion in $\mathrm{KH}^*(K;\Z)$.

Beyond these immediate questions, it is natural to ask which tools from spectral and discrete geometry extend to the Khovanov setting. Hodge Laplacians appear in the definition of combinatorial curvature on simplicial complexes, where discrete Gauss--Bonnet theorems recover the Euler characteristic in low dimensions~\cite{Watanabe2019}. The naive translation of this curvature to Khovanov homology does not satisfy a Gauss--Bonnet-style identity, but one might hope to manipulate the Khovanov Laplacians to produce a notion of curvature of a knot diagram whose Gauss--Bonnet integral recovers the Jones polynomial. In a similar vein, the Laplacian spectrum of a simplicial complex controls random walks on the complex~\cite{Mukherjee2016} and yields Cheeger-type inequalities~\cite{Steenbergen2014}, both rooted in classical results from graph theory~\cite{Fiedler1973, chung1997spectral}. Reinterpreting these results for Khovanov homology requires nontrivial work because of the signs on the cube of resolutions, but the resulting theory would describe diffusion across enhanced states.

Finally, many classical applications of torsion, such as Reidemeister's classification of lens spaces~\cite{Reidemeister1935}, rely on twisting chain complexes to make them acyclic, with the twisting encoding information about the fundamental group and universal cover. A suitable notion of fundamental group and universal covering space for the Khovanov chain complex of a knot diagram would allow more intricate torsion invariants and, in particular, extend Proposition~\ref{prop:torsion of reidemeister moves} to arbitrary knot cobordisms.

%\bibliographystyle{amsplain}
%\bibliography{bibharmonic}

\begin{thebibliography}{10}

\bibitem{Bar_Natan_2005}
D.~Bar-Natan, \emph{Khovanov’s homology for tangles and cobordisms}, Geometry \& Topology \textbf{9} (2005), no.~3, 1443–1499, \href{https://arxiv.org/abs/math/0410495}{arXiv:math/0410495}.

\bibitem{BismutZhang1992}
J.-M. Bismut and W.~Zhang, \emph{An extension of a theorem by {C}heeger and {M}\"uller}, Ast\'erisque, no. 205, Soci\'et\'e Math\'ematique de France, 1992, With an appendix by Fran\c{c}ois Laudenbach.

\bibitem{MR2882891}
A.~E. Brouwer and W.~H. Haemers, \emph{Spectra of graphs}, Universitext, Springer, New York, 2012.

\bibitem{Cheeger1979}
J.~Cheeger, \emph{Analytic torsion and the heat equation}, Annals of Mathematics \textbf{109} (1979), no.~2, 259--321.

\bibitem{chung1997spectral}
F.~R.~K. Chung, \emph{Spectral graph theory}, CBMS Regional Conference Series in Mathematics, vol.~92, American Mathematical Society, Providence, RI, 1997.

\bibitem{CHUNG2009}
J.-W. Chung and X.-S. Lin, \emph{Torsion of quasi-isomorphisms}, Journal of Knot Theory and Its Ramifications \textbf{18} (2009), no.~09, 1227–1258, \href{https://arxiv.org/abs/math/0608459}{arXiv:math/0608459}.

\bibitem{Desai}
M.~Desai and V.~Rao, \emph{A characterization of the smallest eigenvalue of a graph}, Journal of Graph Theory \textbf{18} (1994), no.~2, 181--194.

\bibitem{Eckmann1945}
B.~Eckmann, \emph{Harmonische funktionen und randwertaufgaben in einem komplex}, Commentarii Mathematici Helvetici \textbf{17} (1945), 240--255.

\bibitem{Fiedler1973}
M.~Fiedler, \emph{Algebraic connectivity of graphs}, Czechoslovak Mathematical Journal \textbf{23} (1973), no.~2, 298--305.

\bibitem{Franz1935}
W.~Franz, \emph{{\"U}ber die {T}orsion einer {\"u}berdeckung}, Journal f\"ur die reine und angewandte Mathematik \textbf{173} (1935), 245--254.

\bibitem{MR1622290}
J.~Friedman, \emph{Computing {B}etti numbers via combinatorial {L}aplacians}, Algorithmica \textbf{21} (1998), no.~4, 331--346. \MR{1622290}

\bibitem{HAEMERS2004199}
W.~H. Haemers and E.~Spence, \emph{Enumeration of cospectral graphs}, European Journal of Combinatorics \textbf{25} (2004), no.~2, 199--211, In memory of Jaap Seidel.

\bibitem{Jones2025}
B.~Jones and G.-W. Wei, \emph{Khovanov laplacian and khovanov dirac for knots and links}, Journal of Physics: Complexity \textbf{6} (2025), no.~2, 025014, \href{https://arxiv.org/abs/2411.18841}{arXiv:2411.18841}.

\bibitem{Khovanov2000}
M.~Khovanov, \emph{A categorification of the jones polynomial}, Duke Mathematical Journal \textbf{101} (2000), no.~3, 359--426, \href{https://arxiv.org/pdf/math/9908171}{arXiv:math/9908171}.

\bibitem{L2013}
T.~T.~Q. L\^e, \emph{Growth of regulators in finite abelian coverings}, Algebraic \& Geometric Topology \textbf{13} (2013), no.~4, 2383–2404, \href{https://arxiv.org/abs/1212.6777}{arXiv:1212.6777}.

\bibitem{Lee2005}
E.~S. Lee, \emph{An endomorphism of the khovanov invariant}, Advances in Mathematics \textbf{197} (2005), no.~2, 554–586, \href{https://arxiv.org/abs/math/0210213}{arXiv:math/0210213}.

\bibitem{Liu}
J.~Liu, J.~Li, and J.~Wu, \emph{The algebraic stability for persistent {L}aplacians}, Homology, Homotopy and Applications \textbf{26} (2024), no.~2, 297–323, \href{https://arxiv.org/abs/2302.03902}{arXiv:2302.03902}.

\bibitem{mnev2014lecturenotestorsions}
P.~Mnev, \emph{Lecture notes on torsions}, 2014, \href{https://arxiv.org/abs/1406.3705}{arXiv:1406.3705}.

\bibitem{Mukherjee2016}
S.~Mukherjee and J.~Steenbergen, \emph{Random walks on simplicial complexes and harmonics}, Random Structures \& Algorithms \textbf{49} (2016), no.~2, 379–405, \href{https://arxiv.org/abs/1310.5099}{arXiv:1310.5099}.

\bibitem{Muller1978}
W.~M{\"u}ller, \emph{Analytic torsion and {$R$}-torsion of riemannian manifolds}, Advances in Mathematics \textbf{28} (1978), no.~3, 233--305.

\bibitem{Mmoli2022}
F.~Mémoli, Z.~Wan, and Y.~Wang, \emph{Persistent {L}aplacians: Properties, algorithms and implications}, SIAM Journal on Mathematics of Data Science \textbf{4} (2022), no.~2, 858–884.

\bibitem{ORTIZNAVARRO2012}
J.~Ortiz-Navarro, \emph{A volume form on the khovanov invariant}, Journal of Knot Theory and Its Ramifications \textbf{21} (2012), no.~04, 1250032, \href{https://arxiv.org/abs/0812.0151}{arXiv:0812.0151}.

\bibitem{Rasmussen2010}
J.~Rasmussen, \emph{Khovanov homology and the slice genus}, Inventiones mathematicae \textbf{182} (2010), no.~2, 419–447, \href{https://arxiv.org/abs/math/0402131}{arXiv:math/0402131}.

\bibitem{RaySinger1971}
D.~B. Ray and I.~M. Singer, \emph{{$R$}-torsion and the {L}aplacian on riemannian manifolds}, Advances in Mathematics \textbf{7} (1971), no.~2, 145--210.

\bibitem{Reidemeister1935}
K.~Reidemeister, \emph{Homotopieringe und linsenr\"{a}ume}, Abhandlungen aus dem Mathematischen Seminar der Universit\"{a}t Hamburg \textbf{11} (1935), no.~1, 102–109.

\bibitem{Romero2026}
A.~Romero, \emph{From homological perturbation to spectral sequences: a case study}, 2006, \href{https://www.unirioja.es/cu/anromero/GIFT06.pdf}{unirioja.es}.

\bibitem{Sazdanovi2022}
R.~Sazdanović and D.~Scofield, \emph{Extremal khovanov homology and the girth of a knot}, Journal of Knot Theory and Its Ramifications \textbf{31} (2022), no.~13, 2250083, \href{https://arxiv.org/abs/2003.05074}{arXiv:2003.05074}.

\bibitem{schmidhuber2025quantumalgorithmkhovanovhomology}
A.~Schmidhuber, M.~Reilly, P.~Zanardi, S.~Lloyd, and A.~D. Lauda, \emph{A quantum algorithm for khovanov homology}, 2025, \href{https://arxiv.org/abs/2501.12378}{arXiv:2501.12378}.

\bibitem{Silver2019}
D.~S. Silver and S.~G. Williams, \emph{Knot invariants from laplacian matrices}, Journal of Knot Theory and Its Ramifications \textbf{28} (2019), no.~09, 1950058, \href{https://arxiv.org/abs/1809.06492}{arXiv:1809.06492}.

\bibitem{stacks-project}
{Stacks Project Authors}, \emph{\textit{Stacks Project}}, \url{https://stacks.math.columbia.edu}, 2018.

\bibitem{Steenbergen2014}
J.~Steenbergen, C.~Klivans, and S.~Mukherjee, \emph{A cheeger-type inequality on simplicial complexes}, Advances in Applied Mathematics \textbf{56} (2014), 56–77, \href{https://arxiv.org/abs/2302.01069}{arXiv:2302.01069}.

\bibitem{https://doi.org/10.1112/jtopol/jtv027}
Zoltán Szabó, \emph{A geometric spectral sequence in khovanov homology}, Journal of Topology \textbf{8} (2015), no.~4, 1017--1044.

\bibitem{Turaev2001}
V.~Turaev, \emph{Introduction to combinatorial torsions}, Birkh\"{a}user Basel, 2001.

\bibitem{WangNguyenWei2020}
R.~Wang, D.~D. Nguyen, and G.-W. Wei, \emph{Persistent spectral graph}, International Journal for Numerical Methods in Biomedical Engineering \textbf{36} (2020), no.~9, e3376.

\bibitem{Watanabe2019}
K.~Watanabe, \emph{Combinatorial {R}icci curvature on cell-complex and {G}auss-{B}onnnet theorem}, Tohoku Mathematical Journal \textbf{71} (2019), no.~4.

\bibitem{WeiWei2025}
X.~Wei and G.-W. Wei, \emph{Persistent topological {L}aplacians---a survey}, Mathematics \textbf{13} (2025), no.~2, 208.

\end{thebibliography}

providecommand{\bysame}{\leavevmode\hbox to3em{\hrulefill}\thinspace}
\providecommand{\MR}{\relax\ifhmode\unskip\space\fi MR }
\providecommand{\MRhref}[2]{%
  \href{http://www.ams.org/mathscinet-getitem?mr=#1}{#2}
}
\providecommand{\href}[2]{#2}

\end{document}